\documentclass[11pt]{article}

\usepackage{a4wide,amsfonts,amssymb,stmaryrd,amscd,amsmath,latexsym,amsbsy}
\usepackage{pstricks,pst-plot,pst-node}

\usepackage{amsmath,amsthm,amsfonts}

\theoremstyle{plain}

\newtheorem{theorem}{Theorem}[section]

\newtheorem*{conjecture}{Conjecture}

\newtheorem{thm}{Theorem}[section]

\newtheorem{prop}[theorem]{Proposition}

\newtheorem{corstar}[theorem]{Corollary*}

\newtheorem{lem}[theorem]{Lemma}

\newtheorem{cor}[theorem]{Corollary}

\newtheorem*{sol}{Solution}

\theoremstyle{definition}

\theoremstyle{remark}

\newcommand{\CC}{\mathbb C}
\newcommand{\QQ}{\mathbb Q}

\newcommand{\ZZ}{\mathbb Z}
\newcommand{\HH}{H}

\newcommand{\solu}[1]{\begin{sol}{\bf (\ref{#1})}}

\newcommand{\gotg}{\widehat{\mathfrak g}}

\newcommand{\An}{\mathcal{A}_n}

\newcommand{\pp}{\mathfrak{p}}

\newcommand{\AAA}{\mathcal{A}}

\newcommand{\Mbar}{\overline{M}}

\newcommand{\ev}{\mathrm{ev}}

\newcommand{\hoth}{\hat{\mathfrak{h}}}

\newcommand{\Hb}{\mathrm{Hilb}}

\begin{document}

\title{Quantum cohomology of the Hilbert scheme of points on $\An$-resolutions}

\author{Davesh Maulik and Alexei Oblomkov}
\maketitle
\date{}

\abstract{ We determine the two-point invariants of the equivariant quantum cohomology of the
Hilbert scheme of points of surface resolutions associated to type $A_{n}$ singularities.
The operators encoding these invariants are expressed in terms of the action of the the
affine Lie algebra $\widehat{\mathfrak{gl}}(n+1)$ on its basic representation. Assuming a
certain nondegeneracy conjecture, these operators determine the full structure of the
quantum cohomology ring.  A relationship is proven between the quantum cohomology and
Gromov-Witten/Donaldson-Thomas theories of $A_{n}\times\mathbf{P}^1$.  We close with a
discussion of the monodromy properties of the associated quantum differential equation
and a generalization to singularities of type $D$ and $E$.}

\tableofcontents

\section{Introduction}\label{introduction}

\subsection{Overview}
Given a quasiprojective surface $S$, the Hilbert scheme
$\mathrm{Hilb}_{m}(S)$ of $m$ points on $S$ parametrizes
zero-dimensional subschemes of $S$ of length $m$.  The Hilbert
scheme is a nonsingular irreducible quasiprojective algebraic
variety of dimension $2m$.  It contains an open dense set
parametrizing configurations of $m$ distinct points and can be
viewed as a crepant resolution of the symmetric product of $S$.
The classical cohomology of these varieties has been well-studied
\cite{nakajima, grojnowski, qin-wang} and, as we shall explain later, admits
a description in terms of the representation theory of Heisenberg
algebras.

In this paper, we consider the following family of surfaces. Let
$\zeta$ be a primitive $(n+1)$-th root of unity and let the generator
of the cyclic group $\mathbb{Z}_{n+1}$ act on $\mathbb{C}^2$ by
$$(z_1,z_2) \mapsto (\zeta z_1,\zeta^{-1} z_2).$$ We denote by
 $\An$ the minimal resolution of the quotient
$$\An \rightarrow \mathbb{C}^2/\mathbb{Z}_{n+1}.$$
The diagonal action of $T= (\mathbb{C}^{\ast})^{2}$ on
$\mathbb{C}^2$ commutes with the cyclic group action and
therefore lifts to a $T$-action on both $\An$ and $\Hb(\An)$.

Our goal is to study the small quantum product on the
$T$-equivariant cohomology of $\Hb_m(\An)$ for all $m$ and $n$.
Quantum cohomology is a deformation of the classical cohomology
ring of $\Hb(\An)$.  The structure constants of the ring are
defined by a virtual count of rational curves passing through
specified subvarieties of the Hilbert scheme, weighted by degree.

The main theorem of this paper is an explicit operator formula for quantum
multiplication by divisors in $H_{T}^{2}(\Hb_m(\An),\mathbb{Q})$.
These operators have a simple expression in terms
of the action of the affine Lie algebra $\widehat{\mathfrak{gl}}(n+1)$
on its basic representation.  Under the assumption of a nondegeneracy conjecture (see section~\ref{gener}), 
the divisor operators generate the entire quantum cohomology ring and, in particular,
can be used to calculate the full genus $0$ Gromov-Witten theory.


In the case of $\CC^{2}$, the quantum cohomology ring has been completely calculated in \cite{okpanhilb}.  Along with our results here, these surfaces are the only surfaces for which the divisor operators have been fully calculated for an arbitrary number of points.  Our strategy is motivated by the approach of \cite{okpanhilb} for $\CC^{2}$.  The presence of compact directions on the underlying surfaces greatly complicates the geometry of the Hilbert scheme; the existence of a holomorphic symplectic form is essential for circumventing these difficulties.

\subsection{Relation to other theories}

Using the results of \cite{gwan}, \cite{dtan}, we show that the divisor
operators of this paper satisfy a 
triangle of equivalences between the Gromov-Witten theory of
$\An\times \mathbf{P}^1$, the Donaldson-Thomas theory of
$\An\times \mathbf{P}^1$, and the quantum cohomology of the
Hilbert scheme of points on the $\An$ surface.

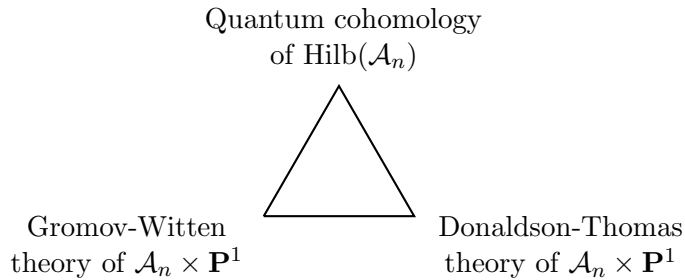
\begin{figure}[hbtp]\psset{unit=0.5 cm}
  \begin{center}
    \begin{pspicture}(-6,-2)(10,6)
    \psline(0,0)(2,3.464)(4,0)(0,0)
    \rput[rt](0,0){
        \begin{minipage}[t]{3.64 cm}
          \begin{center}
            Gromov-Witten \\ theory of $\An \times \mathbf{P}^1$
          \end{center}
        \end{minipage}}
    \rput[lt](4,0){
        \begin{minipage}[t]{3.64 cm}
          \begin{center}
             Donaldson-Thomas\\ theory of $\An \times \mathbf{P}^1$
          \end{center}
        \end{minipage}}
    \rput[cb](2,4.7){
        \begin{minipage}[t]{4 cm}
          \begin{center}
           Quantum cohomology \\ of $\Hb(\An)$
          \end{center}
        \end{minipage}}
    \end{pspicture}
  \end{center}
\end{figure}

The above triangle was first shown to hold for $\mathbf{C}^2$ in
\cite{localcurves},\cite{okpanhilb},\cite{okpandt}.  While the
equivalence between Gromov-Witten theory and Donaldson-Thomas
theory is expected to hold for arbitrary threefolds, the
relationship with the quantum cohomology of the Hilbert scheme
breaks down for a general surface, at least in the specific form
we describe here.  Our work for $\An$ surfaces provides the only
other examples for which this triangle is known to hold.

\subsection{Acknowledgements}

We wish to thank A. Okounkov and R. Pandharipande for many conversations about their paper \cite{okpanhilb}, which were important in shaping the arguments here. We also thank
J. Bryan, P. Etingof, S. Loktev, for many useful discussions.
D.M. was partially supported by an NSF Graduate Fellowship and a Clay Research Fellowship.
A.O. was partially supported by NSF grant DMS-0111298 and DMS-0701387.

\section{Statement of theorems}\label{statements}

In this section, we state precisely the operators for
multiplication by divisors on the quantum cohomology ring.
In order to express these cleanly, we explain how to write
the cohomology of $\Hb(\An)$ in terms of the
representation theory of $\widehat{\mathfrak{gl}}(n+1)$.

\subsection{Topology of $\An$ and the root lattice}

We first set notation for the geometry of the $\An$ surfaces.
Viewed as a crepant resolution of a quotient singularity, the
exceptional locus of $\An$ consists of a chain of $n$ rational
curves $E_1, \dots, E_n$ with intersection matrix given by the
negative Cartan matrix for the Dynkin diagram $A_n$.  That is,
each $E_{i}$ has self-intersection $-2$ and intersects $E_{i-1}$
and $E_{i+1}$ transversely.  These classes span $\HH^2(\An,
\mathbb{Q})$ and, along with the identity class, span the full
cohomology ring of $\An$.  We will also work with the dual basis
$\{\omega_1,\dots,\omega_n\}$ of $\HH^{2}(\An,\QQ)$, defined by the
property that
$$\langle\omega_{i},E_{j}\rangle = \delta_{i,j}$$
under the Poincare pairing.

Under the $T$-action, there are $n+1$ fixed points $p_1, \dots,
p_{n+1}$; the tangent weights at the fixed point $p_i$ are given
by
\begin{gather*}
w^L_i:=(n+2-i)t_1+(1-i)t_2,\\
w^R_i:=(-n+i-1)t_1+it_2.
\end{gather*}
The $E_i$ are the $T$-fixed curves joining $p_{i}$ to $p_{i+1}$.
We denote by $E_{0}$ and $E_{n+1}$ for the noncompact $T$-fixed
curve direction at $p_{1}$ and $p_{n+1}$ respectively.
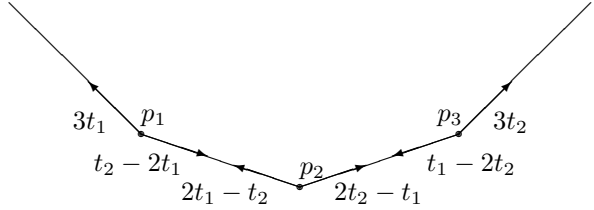
\begin{center}
\begin{picture}(300,130)(0,0) \put(150,30){\line(-3,1){60}}
\put(90,50){\line(-1,1){50}} \put(64,52){$3t_1$}
\put(90,50){\vector(-1,1){20}} \put(90,50){\vector(3,-1){25}}
\put(150,30){\vector(-3,1){25}} \put(72,36){$t_2-2t_1$}
\put(105,25){$2t_1-t_2$} \put(90,50){\circle{2}}
\put(90,50){\circle{1}} \put(90,55){$p_1$}
\put(150,30){\circle{1}}\put(150,30){\circle{2}}
\put(150,35){$p_2$}
\put(150,30){\line(3,1){60}}\put(150,30){\vector(3,1){25}}
\put(210,50){\vector(-3,-1){25}} \put(210,50){\line(1,1){50}}
\put(210,50){\vector(1,1){20}}\put(210,50){\circle{1}}
\put(210,50){\circle{2}} \put(202,55){$p_3$} \put(223,52){$3t_2$}
\put(163,25){$2t_2-t_1$}\put(198,36){$t_1-2t_2$}
\end{picture}
\end{center}

As usual, we denote by $\mathfrak{gl}(n+1)$ the Lie algebra (defined over $\QQ$) of
$(n+1)\times(n+1)$ matrices.  We will denote by $e_{ij} \in
\mathfrak{gl}(n+1)$ the matrix with $1$ in position $(i,j)$ and
$0$ everywhere else.  With respect to the Cartan subalgebra
$\mathfrak{h}$ of diagonal matrices, the roots are given by
functionals $\alpha_{ij} \in \mathfrak{h}^{\ast}$ which take the
value $a_{ii}-a_{jj}$ on the diagonal matrix $(a_{kk})$.  There
is an identification of lattices between $\HH_{2}(\An, \mathbb{Z})$
with the $A_n$ root lattice obtained by sending the exceptional
curves $E_{i}$ to the simple roots $\alpha_{i,i+1}$. Under this
identification, there is a distinguished set of effective curve
classes
$$\alpha_{ij} = E_{i}+ \dots + E_{j-1}$$
which correspond to positive roots in the $A_{n}$ lattice.

\subsection{Fock space formalism}

In this section, we introduce the Fock space modelled on
$\HH_{T}^{\ast}(\An,\QQ)$ and explain the identification with
$\HH_T^{\ast}(\Hb(\An),\QQ)$.  These constructions for projective surfaces
were first provided in \cite{nakajima,grojnowski} and extended in this setting
to the equivariant context in \cite{qin-wang}.
Consider the Heisenberg algebra
$\mathcal{H}$ generated over the field $\QQ(t_1,t_2)$ by a central
element $c$ and elements
$$\mathfrak{p}_{k}(\gamma), \gamma \in \HH_{T}^{\ast}(\An, \QQ), k \in \ZZ, k \ne 0,$$
so that $\mathfrak{p}_{k}(\gamma)$ are $\QQ(t_1,t_2)$-linear in the labels
$\gamma$. The Lie algebra structure on $\mathcal{H}$ is defined by
the following commutation relations
\begin{gather*}
 [\pp_k(\gamma_{1}), \pp_l(\gamma_{2})] = -k \delta_{k+l}\langle\gamma_1, \gamma_2\rangle
\cdot c \\
 [c, \pp_k(\gamma)] = 0.
\end{gather*}
Our sign convention for the commutator differs from \cite{nakajima} but simplifies later formulas.
Notice that we can pick a basis of cohomology for which the Poincare pairing never takes denominators that are divisible by $(t_1+t_2)$; as a result,
$\mathcal{H}$ can be defined over the ring 
$$R = \QQ[t_1,t_2]_{(t_1+t_2)}$$ of rational functions
with non-negative valuation at $(t_1+t_2)$.

The Fock space $\mathcal{F}_{\An}$ is freely generated over the
ring $\QQ[t_1,t_2]$ by the action of the commuting creation
operators $\pp_{-k}(\gamma)$, for $k>0$ and $\gamma \in
\HH^{\ast}_{T}(\An, \QQ)$ on the vacuum vector
$v_{\emptyset}.$  There is an orthogonal grading
$$\mathcal{F}_{\An} = \bigoplus_{m\geq 0} \mathcal{F}_{\An}^{(m)}$$
induced by defining the degree of $v_{\emptyset}$ to be zero and
the degree of each operator $\pp_k(\gamma)$ to be $-k$.

After extension of scalars to the fraction field $\QQ(t_1,t_2)$,
$\mathcal{F}_{\An}$ admits the structure of a representation of
$\mathcal{H}$.  Under this action, the annihiliation operators
$\pp_k(\gamma)$ with $k > 0$ kill the vacuum vector and the
central element $c$ acts trivially. Similarly, we define a
nondegenerate pairing on $\mathcal{F}_{\An}\otimes \QQ(t_1,t_2)$
by requiring
$$\langle v_{\emptyset}| v_{\emptyset}\rangle = 1$$
and specifiying the adjoint
$$\pp_k(\gamma)^{\ast} = (-1)^{k}\pp_{-k}(\gamma).$$

There is a graded isomorphism
$$\mathcal{F}_{\An} = \bigoplus_{m \geq 0} H_{T}^{\ast}(\Hb_m(\An),\QQ)$$
Under this isomorphism, the Heisenberg operators $\pp_k(\gamma)$
are defined by correspondences between Hilbert schemes of
different numbers of points. The inner product on
$\mathcal{F}_{\An}$ over $\QQ(t_1,t_2)$ defined above corresponds
to the Poincare pairing on $\Hb_m(\An)$ defined by $T$-equivariant
residue.

\subsection{Nakajima basis}\label{Nak basis}

If we work with a fixed basis $\{\gamma_{0}, \dots, \gamma_{n}\}$
of $H_{T}^{\ast}(\An,\QQ)$, our Fock space has a natural basis
indexed as follows.  Given a nonnegative integer $m$, a
\textit{cohomology-weighted} partition of $m$ consists of an
unordered set of pairs
$$\overrightarrow{\mu} = \{(\mu^{(1)}, \gamma_{i_1}), \dots,
(\mu^{(l)}, \gamma_i^{l})\}$$ where
$\{\mu^{(1)},\dots,\mu^{(l)}\}$ is a partition whose parts are
labelled by elements $\gamma_{i_{k}}$ in our basis.  For each
cohomology-weighted partition $\overrightarrow{\mu}$ as above, the
associated basis element is given by
$$\frac{1}{\mathfrak{z}(\overrightarrow{\mu})}
\prod_{k=1}^{l} \pp_{- \mu^{(k)}}({\gamma_{i_{k}}})
v_{\emptyset},$$
where 
$$\mathfrak{z}(\overrightarrow{\mu}) = \prod \mu^{(k)} \cdot |\mathrm{Aut}(\overrightarrow{\mu})|$$
We will also denote this basis element by
$\overrightarrow{\mu}$ when there is no confusion.  
The associated
basis of $H_{T}^{\ast}(\Hb_m(\An),\QQ)$ will be called the Nakajima
basis. It is clear from the definition that this basis respects
the grading of $\mathcal{F}_{\An}$, so that cohomology-weighted
partitions of $m$ form a basis of $H_{T}^{\ast}(\Hb_{m}(\An),\QQ)$
over $\QQ[t_1,t_2]$.  If we want to stress the underlying choice of basis $\gamma_{k}$, we will write 
$$\mu_{1}(\gamma_{1})\cdot\mu_{2}(\gamma_{2})\dots\mu_{r}(\gamma_{r})$$
where $\mu_{k}$ is the subpartition of $\mu$ labelled with $\gamma_{k}$.

The cohomological degree of a basis element
$\overrightarrow{\mu}$ under these identifications is
$$2(m - l(\mu)) + \sum \mathrm{deg}(\gamma_{i_{k}}).$$
Finally, under the $T$-equivariant Poincare pairing, the dual
basis of the Nakajima basis is given, up to constant factors, by
cohomology-weighted partitions labelled with the dual basis of
$\{\gamma_{i}\}$.

In terms of the basis $\{1, \omega_1, \dots, \omega_n\}$, we can
see from the formula for cohomological degree that there are two
types of divisors in the Nakajima basis. First, we have
$$D = - \{(2,1), (1,1)^{m-2}\}$$
which is proportional to the boundary divisor on $\Hb_{m}(\An)$
where two points collide. Second, we have for $i = 1, \dots, n$,
$$(1,\omega_{i}) = \{(1,\omega_{i}), (1,1)^{m-1}\}.$$
These latter divisors are clearly non-negative on effective curve classes.
We will give explicit formulas for quantum multiplication by these
elements.  We will also use this basis of divisors to measure
degrees of curve classes.

\subsection{Affine algebra $\widehat{\mathfrak{gl}}(n+1)$ and the basic representation}

The Fock space description given above can be reinterpreted in
terms of the representation theory of the affine algebra
$\widehat{\mathfrak{gl}}(n+1)$.  While the previous discussion is
valid for any surface $S$, this interpretation is a special
feature of the $\An$ surfaces.

The Lie algebra $\gotg = \widehat{\mathfrak{gl}}(n+1)$ is defined
over $\QQ$
in terms of a central extension of the loop algebra
$\mathfrak{gl}(n+1)\otimes \QQ[t, t^{-1}]$. It is generated by
elements
$$x(k)= x\cdot t^{k}, \quad x \in \mathfrak{gl}(n+1),\quad k \in \ZZ,$$
a central element $c$ and a differential $d$.  The defining
relations are
\begin{gather*}
[x(k),y(l)]=[x,y](k+l)+k\delta_{k+l,0}tr(xy)c,\\
[d,x(k)]=kx(k),\quad [d,c]=0,
\end{gather*}
where $tr(xy)$ refers to the trace of the matrix $xy$. The Cartan
subalgebra of $\gotg$ is given by
$$\hoth =\mathfrak{h} \oplus\QQ c\oplus \QQ d.$$

Using this direct sum decomposition, we can write the dual space
$$\hoth^{\ast} = \mathfrak{h}^{\ast} \oplus \QQ \Lambda \oplus \QQ \delta$$
where $\Lambda(c) = 1, \Lambda(d) = 0$ and $\delta(d) = 1,
\delta(c) = 0$. As in the finite-dimensional situation, there is a
theory of roots and weights lying in this dual space. The roots of
$\gotg$ are given by
$$\Delta = \{k \delta + \alpha_{i,j}, k \in \ZZ\} \cup \{ k \delta, k \neq 0\}.$$
Associated to the weight $\Lambda$, there is a unique irreducible
highest weight representation $(V_{\Lambda},\rho)$,
 containing a vector
$v$ such that
$$\rho(\mathfrak{gl}(n+1)\otimes\QQ[t])v = 0$$ and
$$\rho(c)v = v$$
This is known as the {\it basic representation} of $\gotg$.

We have an embedding of Lie algebras of
the Heisenberg algebra $\mathcal{H}$ associated to $\An$
$$\mathcal{H} \rightarrow \gotg \otimes \QQ(t_1,t_2)$$
given by the map
\begin{gather*}
\pp_{-k}(1)\mapsto \mathrm{Id}(-k), \quad
\pp_{k}(1)\mapsto -\mathrm{Id}(k)/((n+1)^2t_1t_2),\quad k>0,\\
\pp_k(E_i)\mapsto e_{i,i}(k)- e_{i+1,i+1}(k),\quad c\mapsto 1.
\end{gather*}
 This embedding is compatible with the identification of $H_{2}(\An, \ZZ)$
 with the finite
 $A_{n}$ root lattice.  We identify $\mathcal{H}$ with
its image inside $\gotg\otimes \QQ(t_1,t_2)$.

Rather than study the full representation $V_{\Lambda}$, we will
work with the subspace
$$W = \bigoplus_{m\geq 0} V_{\Lambda}[\Lambda - m\delta]$$
where $V[\alpha]$ denotes the weight space of $V$ associated to
the weight $\alpha$. By construction, $W$ is graded by
non-negative integers.  Moreover, after extension of scalars, the
space $W \otimes \QQ(t_1,t_2)$ is preserved by the action of
$\mathcal{H}$.  We then have the following easy observation.

\begin{prop}
There is an isomorphism of $\mathcal{H}$-modules
$$W \otimes \QQ(t_1,t_2) \rightarrow \mathcal{F}_{\An} \otimes \QQ(t_1,t_2) $$
uniquely specified by requiring $v$ to map to $v_{\emptyset}$.
\end{prop}

Although we will not need it, Nakajima has identified the
entire basic representation of $\gotg$ with the cohomology of
various moduli spaces.
Also, note that while we have extended scalars to $\QQ(t_1,t_2)$, all objects and maps
given here are again defined over $R = \QQ[t_1,t_2]_{(t_1+t_2)}$.

\subsection{Quantum cohomology}\label{endofgeneral}

Let $\beta \in H_{2}(\Hb_{m}(\An), \ZZ)$ be a curve class. Given
cohomology classes $\overrightarrow{\mu}_{1}, \dots,
\overrightarrow{\mu}_{k} \in H_{T}^{\ast}(\Hb_{m}(\An),\QQ)$, the
$k$-point genus $0$ Gromov-Witten invariant is defined by
integration agains the $T$-equivariant virtual class of the moduli
space of genus $0$ stable maps $$ \langle
\overrightarrow{\mu_{1}}, \dots,
\overrightarrow{\mu_{k}}\rangle^{\Hb}_{0,k,\beta} =
\int_{[\Mbar_{0,k}(\Hb(\An),\beta]^{\mathrm{vir}}}
\ev_{1}^{\ast}\overrightarrow{\mu_{1}} \dots
\ev_{k}^{\ast}\overrightarrow{\mu_{k}}
$$
where $\ev_{i}:\Mbar_{0,k}(\Hb(\An),\beta) \rightarrow \Hb(\An)$
denote evaluation maps.  We can then encode these invariants in
the $k$-point generating function
$$
\langle \overrightarrow{\mu_{1}}, \dots,
\overrightarrow{\mu_{k}}\rangle^{\Hb} = \sum_{\beta} \langle
\overrightarrow{\mu_{1}}, \dots,
\overrightarrow{\mu_{k}}\rangle^{\Hb}_{0,k,\beta} q^{D\cdot \beta}
\prod_{i=1}^{n}s_{i}^{(1,\omega_{i})\cdot\beta} \in
\QQ(t_1,t_2)((q))[[s_1,\dots,s_n]]$$ where again $D$ and
$(1,\omega_{i})$ are the basis of divisors on $\Hb(\An)$ described
earlier.

Since $\An$ is noncompact, some care is required to define the
above integrals rigorously.  Although
$\Mbar_{0,k}(\Hb(\An),\beta)$ may be noncompact, its $T$-fixed
locus is necessarily compact.  In this case, we can define the
above integral by the pushforward of its equivariant residue.  In
the compact case, this agrees with the original definition by the
virtual localization formula.

In the case of $k=3$, we can define the quantum product on
$$QH_{T}(\Hb_{m}(\An)) = H_{T}^{\ast}(\Hb_{m}(\An),\QQ)\otimes
\QQ(t_1,t_2)((q))[[s_1,\dots,s_n]]$$ as follows.  Given
$\overrightarrow{\mu}, \overrightarrow{\nu},
\overrightarrow{\rho}$, the quantum product $\circ$ is defined
using the structure constants
$$\langle \overrightarrow{\mu}|
\overrightarrow{\rho}\circ\overrightarrow{\nu}\rangle = \langle
\overrightarrow{\mu},
\overrightarrow{\rho},\overrightarrow{\nu}\rangle^{\Hb} \in
\QQ(t_1,t_2)((q))[[s_1,\dots,s_n]].$$ The brackets on the
left-hand side of the above equation denote the inner product on
Fock space.  It is a standard fact that these structure constants
define a ring deformation of the classical equivariant cohomology
ring.

\subsection{Operator formulas}

The main object of study in this paper are the two-point genus $0$ invariants
for $\Hb(\An)$.  These can be encoded in an operator on Fock space; more precisely,  we define the operator $\Theta(q,s_1,\dots,s_n)$ by the equality
\begin{equation}\label{Op Theta}
\langle \overrightarrow{\mu}|\Theta(q,s_1,\dots,s_n)|\overrightarrow{\nu}\rangle
=\langle \overrightarrow{\mu}, \overrightarrow{\nu}\rangle^{\Hb}.
\end{equation} 
In the left-hand side, the angle brackets refer to the inner
product on Fock space induced by Poincare duality.

Consider the following function of
$q,s_1, \dots,s_n$ with coefficients in $U(\gotg)$:
\begin{gather*}
\Omega_{+}:=\sum_{1\le i<j\le n+1}\sum_{k\in\ZZ} :e_{ji}(k)
e_{ij}(-k):\log(1-(-q)^k s_i\dots s_{j-1})
\end{gather*}
In this expression, we use the normal ordering shorthand where
$$
:e_{ji}(k) e_{ij}(-k):=\left\{ \begin{aligned}
e_{ji}(k)e_{ij}(-k),& \quad k<0 \textrm{ or } k=0, i<j \\
e_{ij}(-k) e_{ji}(k),& \quad\textrm{otherwise}.
\end{aligned}
\right.
$$
Moreover, we expand the logarithms so the Taylor expansion has
nonnegative exponents in the $s$ variables.

A priori, the above expression defines an operator on the entire
basic representation $V_{\Lambda}$. However, a direct calculation
shows that each summand $:e_{ij}(k)e_{ji}(-k):$ commutes with
elements of the Cartan subalgebra $\hoth$.  As a result,
$\Omega_{+}$ preserves each weight subspace of $V_{\Lambda}$ and
in particular descends to a well-defined operator on each graded
component
$$V_{\Lambda}[\Lambda-m\delta]\otimes\QQ(t_1,t_2) = H_{T}^{\ast}(\Hb_m(\An),\QQ).$$
Moreover, for each graded piece, it is clear that only finitely
many summands in $\Omega_{+}$ contribute.

We also have the following operator
$$\Omega_{0} = -\sum_{k\geq 1} \left[(n+1)t_{1}t_{2}\pp_{-k}(1)\pp_{k}(1) +
\sum_{i=1}^{n} \pp_{-k}({E_{i}})
\pp_{k}(\omega_{i})\right]\log\left( \frac{1-(-q)^k}{1-(-q)}   \right)$$ and the sum
$$\Omega = \Omega_{0} + \Omega_{+}.$$
The expression for $\Omega_{0}$ is already written in terms of
Heisenberg operators, so it obviously acts on Fock space.

Our main theorem is that $\Omega(q,s_1,\dots,s_n)$ is essentially the operator
encoding two-point Gromov-Witten invariants for $\Hb(\An)$.

\begin{thm}\label{theorem1}
The generating function of two-point invariants of $\Hb_{m}(\An)$
is given by the following operator equality
$$\Theta(q,s_1,\dots,s_n) = (t_1+t_2)\cdot \Omega(q,s_1,\dots,s_n).$$
In terms of two-point invariants, we have
$$\langle \overrightarrow{\mu}, \overrightarrow{\nu}\rangle^{\Hb} =
(t_1+t_2) \langle \overrightarrow{\mu}| \Omega(q,s_{1},\dots,s_{n})|
\overrightarrow{\nu}\rangle.
$$
\end{thm}

As a corollary, we have the following operator expressions for
quantum multiplication by the divisors $D$ and $(1,\omega_i)$. Let
$$M_{D}, M_{(1,\omega_{1})}, \dots, M_{(1,\omega_{n})}$$
denote the operators on $\bigoplus_{m}
QH_{T}^{\ast}(\Hb_{m}(\An))$. Let $M_{D}^{cl},
M_{(1,\omega_{i})}^{cl}$ denote the operators for classical
multiplication.

\begin{cor}\label{div mult}
We have the equality
 \begin{gather*}
M_{D} = M_{D}^{cl} + (t_1+t_2) q \frac{d}{dq}\Omega(q,s_1,\dots,s_n)\\
M_{(1,\omega_{i})} = M_{(1,\omega_{i})}^{cl} +
(t_1+t_2)s_{i}\frac{d}{ds_{i}} \Omega_{+}(q,s_1,\dots,s_n)
\end{gather*}
\end{cor}
\begin{proof}
This follows immediately from the divisor equation
$$\langle \overrightarrow{\mu},D,\overrightarrow{\nu}\rangle_{\beta} = (D\cdot \beta)
\langle \overrightarrow{\mu},\overrightarrow{\nu}\rangle_{\beta}$$
for $\beta \ne 0$.
\end{proof}

\subsection{Gromov-Witten theory of $\An\times\mathbf{P}^1$}

Consider the projective line $\mathbf{P}^1$ with three
distinguished marked points $0,1,\infty$.
 The $T$-equivariant
Gromov-Witten theory of $\An\times\mathbf{P}^{1}$ relative to the
fibers over $0,1,\infty$ has been studied in
\cite{gwan}.
 If we study curve classes of degree $m>0$
over the base $\mathbf{P}^{1}$, then relative conditions at each
fiber are given by cohomology-weighted partitions of $m$
$$\overrightarrow{\mu},\overrightarrow{\nu},\overrightarrow{\rho}.$$
The relative Gromov-Witten theory of $\An\times\mathbf{P}^{1}$ is
encoded by the generating funciton
$$\mathsf{Z}^{\prime}_{GW}(\An\times\mathbf{P}^1)_{\overrightarrow{\mu},\overrightarrow{\nu},\overrightarrow{\rho}} \in \QQ(t_1,t_2)((u))[[s_{1}, \dots,
s_{n}]].$$ The reader should see \cite{gwan} for an
explanation of this generating function; the $u$-variable encodes
the genus of the domain curve while the variables $s_{i}$ again
encode the degree of the curve classes with respect to the
divisors $\omega_{i}$ on the $\An$ surface.

By comparing our formulas with those of \cite{gwan}, we
will prove the following precise version of the
Gromov-Witten/Hilbert correspondence discussed in the
Introduction.

\begin{thm}
Under the variable substitution $q = -e^{iu}$, we have
\begin{align*}
(-1)^{m}\langle\overrightarrow{\mu},\overrightarrow{\nu},\overrightarrow{\rho}\rangle^{\mathrm{Hilb}}_{\An}
=
(-iu)^{-m+l(\mu)+l(\nu)+l(\rho)}\mathsf{Z}^{\prime}(\An\times\mathbf{P}^1)_{\overrightarrow{\mu},\overrightarrow{\nu},\overrightarrow{\rho}}
\end{align*}
for $\overrightarrow{\nu} = (2)$ and $\overrightarrow{\nu} = (1,\omega_{i})$.
\end{thm}
Moreover, the statements and proofs in \cite{gwan} were
motivation for the strategy pursued in this paper.  Again assuming
a nondegeneracy conjecture, the above statement should hold for arbitrary three-point invariants.

Just as with quantum cohomology, the three-point invariants
defined using the Gromov-Witten theory of
$\An\times\mathbf{P}^{1}$ can also be used to define a ring
deformation of the classical equivariant cohomology
$H_{T}^{\ast}(\An,\QQ)$ over $\QQ(t_1,t_2)((u))[[s_1,\dots,s_n]]$.
An equivalent version of this correspondence is the statement
that, after a transcendental change of variables, this ring
deformation is {\it explicitly} isomorphic to the quantum
deformation.

In upcoming work \cite{dtan}, we will finish the comparison of
both of these three-point invariants with the Donaldson-Thomas
theory of $\An \times\mathbf{P}^{1}$

\subsection{Overview of proof}

Our operator formulas have many striking qualitative features
which are geometrically surprising and turn out to be useful in
their proof.  First, as a function of $s_{1}, \dots, s_{n}$, we
observe that (\ref{Op Theta}) is essentially root-theoretic.  That is, the
only curve classes from $\An$ that contribute to $\Omega$ are
those corresponding to positive roots $\alpha_{ij}$ and their
multiples.  Moreover, the associated expression is essentially
independent of which root we are working with.

Second, if we fix a root $\alpha_{i,j}$ and isolate the
contribution to $\Omega$ from its multiples, the associated terms
have a logarithmic dependence on $(s_{i}\cdot\dots s_{j-1}).$  In
practice, this means that the contribution from $d\alpha_{i,j}$ is
essentially given by the contribution from $\alpha_{i,j}$.  In
terms of the divisor operators from corollary~\ref{div mult}, this implies
that their matrix entries are rational functions in
$q,s_{1},\dots, s_{n}$.  While one expects rational functions in
$q$ for theoretical reasons, the fact that these operators are
rational functions in all variables is a special feature of these
geometries.

Our proof proceeds by establishing these qualitative features 
directly and using them to algorithmically reduce the computation to a few basic cases.
We then show that both $\Theta$ and $\Omega$ satisfy the
same qualitative properties and have the same value on the
basic cases.  In section $3$, we establish some preliminary lemmas
involving reduced virtual classes.  In section $4$, we 
set up the algorithm on the geometric side for the operator $\Theta$.
In section $5$, we perform the same analysis for $\Omega$ and prove the
main theorem and its corollaries in section $6$.  In particular,
we discuss a conjecture on $\Omega$ which implies that these divisors generate
the full quantum cohomology ring.
Finally, in section $7$,
we discuss elementary properties of the quantum differential equation and give
an argument, due to Jim Bryan, for extending these formulas to $D$ and $E$ surface
resolutions.

\section{Preliminary lemmas}

In this section, we set some further notation and explain basic
properties of the reduced virtual class for $\Hb(\An)$.

\subsection{Definitions}

We first clarify the definition of two-point Gromov-Witten
invariants
$$\langle \overrightarrow{\mu}, \overrightarrow{\nu} \rangle^{\Hb}_{\beta} =  
\int_{[\Mbar_{0,2}(\Hb(\An),\beta]^{\mathrm{vir}}}
\ev_{1}^{\ast}\overrightarrow{\mu} 
\ev_{2}^{\ast}\overrightarrow{\nu}
.$$
The integrand is the virtual fundamental class on the space of
stable maps and has dimension
$$-K_{\Hb(\An)}\cdot \beta + (2n-3) +2 = 2n-1.$$

As mentioned earlier, the space of stable maps is typically
noncompact. There are two approaches to making sense of the above
expression.  The first approach, already described, is to use the
fact that its $T$-fixed locus is compact.  The equivariant residue
is defined to be the formal contribution of these fixed loci to
virtual localization.  We can take the integral to be the
pushforward to a point of these equivariant residues, which will
take values in $\QQ[t_1,t_2]_{(t_{1}+t_{2})}$ due to the denominators occuring in 
the residue expressions.

Alternately, if the insertions correspond to $T$-equivariant
cycles for which the space of maps meeting these cycles is
compact, then this integral can be replaced with one over this
incidence locus.  We then have a $T$-equivariant integral on a
compact space which is defined in the usual sense. In particular,
the associated invariant takes values in $\QQ[t_1,t_2]$. Since any
insertion can be written as a combination of these compact cycles
with coefficients in $\QQ(t_1,t_2)$, we can use this as a
definition in general.  This agrees with the first approach, by the
virtual localization formula.

The Hilbert-Chow morphism defines a map from the Hilbert scheme of
points to the symmetric product of $\An$:
$$\rho^{HC}: \Hb_{m}(\An) \rightarrow S^{m}(\An).$$
We distinguish curve classes $\beta \in H_{2}(\Hb(\An),\ZZ)$ based
on whether they are contracted by $\rho^{HC}_{\ast}$.  Curves that
are contracted by $\rho^{HC}$ will be called {\it punctual
curves}, since they parametrize subschemes with fixed support.
Noncontracted curves will be called nonpunctual curves.  It is
easy to see that a class $\beta$ is punctual if and only if
$(1,\omega_{i})\cdot\beta = 0$ for all $i$ and is effective if, in
addition, $D \cdot \beta \geq 0.$

As we shall explain later, the contribution to the two-point operator
$\Theta$ from punctual curve classes
 can be deduced from the calculations for $\Hb(\mathbb{C}^{2})$ in \cite{okpanhilb}.
For most of this paper, we will study the contribution to $\Theta$ from
non-punctual curves, denoted by
$$\langle \vec{\mu}|\Theta_{+}(q,s_{1},\dots,s_{n})| \vec{\nu} \rangle\in \QQ(t_1,t_2)((q))[[s_{1},\dots,s_{n}]].$$
Moreover, given a curve class $\alpha = c_{1}E_{1} +\dots+ c_{n}E_{n}\in
H_{2}(\An)$, we isolate the coefficient in $\Theta$ of
$s_{1}^{c_{1}}\cdot\dots\cdot s_{n}^{c_{n}}$ as
$\Theta_{\alpha}(q)$ so that the corresponding invariants are given by
$$\langle \vec{\mu}|\Theta_{\alpha}(q)|\vec{\nu}\rangle \in \QQ(t_1,t_2)((q)).$$
Notice that we are still considering all possible values of
$D\cdot \beta$, so the result is a Laurent series in $q$.  

Finally, given a curve class $\beta \in H_{2}(\Hb_{m}(\An),\ZZ)$,
we define the support of $\beta$ to be the smallest interval $1
\leq i < j \leq n$ containing the set $\{k | (1,\omega_{k})\cdot
\beta \ne 0\}$.  It will be useful to consider the contribution to $\Theta$ of
all curve classes with fixed support $[i,j]$:
$$\langle \vec{\mu}| \Theta_{[i,j]}| \vec{\nu}\rangle \in \QQ(t_1,t_2)((q))[[s_{i},\dots,
s_{j-1}]].$$

\subsection{Bases for Fock space}\label{bases}

It will be useful in our arguments to shift between different
bases for the cohomology of $\Hb_{m}(\An)$. First, as discussed in
section~\ref{Nak basis}, for every choice of basis $\{\gamma_{k}\}$ of
$H_{T}^{\ast}(\An,\QQ)$, there is an associated Nakajima basis
$\{\vec{\mu}\}$ given by partitions weighed by elements of this
basis. We will work with the Nakajima basis indexed by weighted partitions of the form
$$\mu_{1}(\omega_{1})\dots\mu_{n}(\omega_{n})\mu_{n+1}(1)$$
where $\{\mu_{i}\}$ is an $(n+1)$-tuple of partitions (or multipartition) such that $\sum |\mu_{i}|=m$

If we extend coefficients to $\QQ(t_1,t_2)$, then we have a basis
of $H_{T}^{\ast}(\An,\QQ)\otimes\QQ(t_1,t_2)$ given by fixed
points $[p_{1}],\dots, [p_{n+1}]$ and can also study its associated Nakajima
basis
$$\mu_{1}([p_{1}])\dots\mu_{n+1}([p_{n+1}]),$$
where $\{\mu_{i}\}$ is a again a multipartition of $m$.

Finally, we will also work with the basis of localized equivariant
cohomology given by $T$-fixed points on $\Hb_{m}(\An)$, which can
be described as follows. We first recall the description of
$T$-fixed subschemes of $\mathbb{C}^{2}$ under the standard torus action.
Given such a subscheme of length $m$, it must be the zero locus of
a monomial ideal
$$I_{\lambda} =(x^{\lambda_1},yx^{\lambda_2},\dots,y^{l-1}x^{\lambda_l})$$
associated to the partition $\lambda$ of $m$.

The toric surface $\An$ admits an affine cover by open sets $U_k
\cong \mathbb{C}^{2}$ centered at the fixed point $p_{k}$, where we fix
the identification with $\mathbb{C}^{2}$ so that the $x$ and $y$-axes
correspond to $E_{k-1}$ and $E_{k}$ respectively.  Given a
$T$-fixed subscheme of length $m$ on $\An$, its restriction to
each $U_{k}$ yields an associated monomial ideal and partition.
This gives a bijection between $T$-fixed subschemes of length $m$
and multipartitions $\vec{\lambda} = \{\lambda_{k}\}$
such that $\sum |\lambda_{k}| = m$. The associated cohomology
class shall be denoted by
$$[J_{\vec{\lambda}}]\in\HH_{T}^{\ast}(\Hb_m(\An))\otimes \QQ(t_1,t_2).$$

The relationship between the Nakajima basis associated to $\{p_{i}\}$
and the fixed-point basis can be described in terms of symmetric
functions as follows (see \cite{nakjack, li-qin-wang, okpanhilb}). Both of these bases are induced by the
Nakajima and fixed-point bases for $\Hb(\mathbb{C}^{2})$ under the
isomorphism
$$\mathcal{F}_{\An}\otimes\QQ(t_1,t_2) = \bigotimes_{i=1}^{n+1}
\mathcal{F}_{\mathbb{C}^{2},i}\otimes \QQ(w^{i}_{L},w^{i}_{R})$$ where
the coordinate axes on the $i$-th factor have been identified with
the tangent weights at $p_{i}$. Let $\mathrm{Symm}$ denote the
ring of symmetric functions over $\QQ(t_1,t_2)$ in countably many
variables $z_{1},z_{2},\dots.$  This ring admits an isomorphism with Fock space by
identifying $1$ with the vacuum vector $v_{\emptyset}$ and the
Heisenberg operator $\pp_{-k}([p_{i}])$ with multiplication by the
Newton symmetric polynomial $w_{R}^{i}\cdot
\mathsf{p}_{k}(z)$ where $\mathsf{p}_{k}(z) = \sum_{a} z_{a}^{k} \in
\mathrm{Symm}.$ Under this identification, the Nakajima basis element
$\mu_{i}([p_{i}])$ is idenitifed with the Newton symmetric function
$\frac{(w_{R}^{i})^{l(\mu_{i})}}{\mathfrak{z}(\mu_{i})}\mathsf{p}_{\mu_{i}}(z)$ and the normalized
fixed-point $(w_{R}^{i})^{-|\lambda|}[J_{\lambda}]$ is identified with the integral Jack
polynomial
$$\mathsf{J}_{\lambda}^{\theta}(z)$$
with parameter $\theta = w_{L}^{i}/w_{R}^{i}$. If we specialize
our equivariant weights so that $t_1+t_2 = 0$ then the Jack
polynomials written above become proportional to the associated
Schur polynomials
$$\mathsf{s}_{\lambda}(z)=(-1)^{|\lambda|}\frac{\dim\lambda}{|\lambda|! }\mathsf{J}_{\lambda}(z) \mod (t_1+t_2)$$ where $\dim \lambda$ is
the dimension of the irreducible representation of $S_{m}$
associated to $\lambda$.

To apply this to $\An$, the Nakajima basis element $\prod \mu_{i}([p_{i}])$ now corresponds to an
$(n+1)$-tuple of power-sum symmetric functions 
$\otimes_{i}
\mathsf{p}_{\mu_{i}}(z^{(i)})$ and the fixed-point
basis element $[J_{\overrightarrow{\lambda}}]$ corresponds to an $(n+1)$-tuple of Jack symmetric
polynomials $\otimes \mathsf{J}_{\lambda_{i}}(z^{(i)}).$  Under the specialization $t_{1}+t_{2}=0$, we again
are allowed to work with $(n+1)$-tuples of Schur polynomials.

\subsection{Reduced virtual classes}
We now define the reduced virtual fundamental class for
$\Mbar_{0,k}(\Hb(\An),\beta)$.  Given any variety with an
everywhere-nondegenerate holomorphic symplectic form, this form
gives rise to a trivial factor of the obstruction theory.  By
removing this trivial factor by hand, we obtain a new obstruction
theory with with virtual dimension increased by $1$.  In the case
of $\Hb(\mathbb{C}^{2})$, this construction is important in the analysis
of \cite{okpanhilb} and we will use it in much the same way. Our
discussion is based on the more detailed treatment given there.

We first explain the standard and modified obstruction theory for
a fixed domain curve $C$.  Given a fixed nodal, pointed curve $C$
of genus $0$, let $M_{C}(\Hb(\An), \beta)$ denote the moduli
space of maps from $C$ to $\Hb(\An)$ of degree $\beta \neq 0$.  The
usual perfect obstruction theory for $M_C(\Hb(\An),\beta)$ is defined
by the natural morphism
\begin{equation}\label{standardmap}
R\pi_{\ast}(\mathrm{ev}^{\ast}T_{\Hb})^{\vee} \rightarrow
L_{M_{C}},
\end{equation}
where $L_{M_{C}}$ denotes the cotangent complex of
$M_{C}(\Hb(\An),\beta)$ and
\begin{gather*}
\mathrm{ev}: C \times M_C(\Hb(\An),\beta) \rightarrow \Hb(\An),\\
\pi: C \times M_C(\Hb(\An), \beta) \rightarrow M_C(\Hb(\An),\beta).
\end{gather*}
are the evaluation and projection maps.

Let $\gamma$ denote the holomorphic symplectic form on $\An$
induced by the standard form $dx\wedge dy$ on $\mathbb{C}^2$
induces holomorphic symplectic form on $\An$ which, in turn,
induces a holomorphic symplectic form $\gamma$ on $\Hb_{m}(\An)$.
The $T$-representation $\mathbb{C}\cdot \gamma$ has weight
$-(t_1+t_2)$. Let $\omega_{\pi}$ denote the relative dualizing
sheaf.   The symplectic pairing and pullback of
differentials induces a map
$$\mathrm{ev}^{\ast}(T_{\Hb(\An)})\rightarrow
\omega_{\pi}\otimes(\mathbb{C}\gamma)^{\ast}.$$ This, in turn,
yields a map of complexes
$$R\pi_{\ast}(\omega_{\pi})^{\vee}\otimes\mathbb{C}\gamma \rightarrow
R\pi_{\ast}(\mathrm{ev}^{\ast}(T_{\Hb(\An)})^{\vee})$$ and the
truncation
$$\iota:\tau_{\leq -1}R\pi_{\ast}(\omega_{\pi})^{\vee}\otimes\mathbb{C}\gamma \rightarrow
R\pi_{\ast}(\mathrm{ev}^{\ast}(T_{\Hb(\An)})^{\vee}).$$ This truncation
is a trivial line bundle with equivariant weight $-(t_1+t_2)$.

Results of Ran and Manetti (\cite{ran,manetti}) on obstruction
theory and the semiregularity map imply the following.  First,
there is an induced map
\begin{equation}\label{reducedmap}
C(\iota) \rightarrow L_{M_{C}}
\end{equation}
where $C(\iota)$ is the mapping cone associated to $\iota$.
Second, this map (\ref{reducedmap}) satisfies the necessary
properties of a perfect obstruction theory.  This is precisely the
modified obstruction theory we use to define the reduced virtual
class.   Since all maps in this section are compatible with the
$T$-action, we have a $T$-equivariant reduced virtual class.

There is one important subtlety regarding the semiregularity
results of (\cite{ran,manetti}). In order to apply their results,
we require a compact target space.  We can embed the $\An$
singularity in a surface $S$ with a holomorphic symplectic form
that is degenerate away from the singularity.  Given a stable map
$f$, its deformation theory can be studied
 on the Hilbert scheme of the resolved surface $\widetilde{S}$ where
our curve maps entirely to the nondegenerate locus.  Theorem $9.1$
of \cite{manetti} still gives the necessary vanishing statement
for realized obstructions.

As with the standard obstruction theory (\ref{standardmap}), we
obtain the \textit{reduced}  $T$-equivariant perfect obstruction
theory on $\Mbar_{0,k}(\Hb(\An),\beta)$ by varying the domain $C$,
and studying the the relative obstruction theory over the Artin
stack $\mathfrak{M}$ of all nodal curves.  Since the new
obstruction theory differs from the standard one by the
1-dimensional obstruction space $(\mathbb{C}\gamma)^{\vee}$, we have that
the reduced virtual dimension is given by
$$1 + (2n-3) +k.$$
Furthermore we have the identity
\begin{align*}
[\Mbar_{0,k}(\Hb(\An),\beta)]^{vir}_{\mathrm{standard}}
&= c_{1}(\mathbb{C}\gamma^{\vee})[\Mbar_{0,k}(\Hb(\An),\beta)]^{\mathrm{red}}\\
&=(t_1+t_2)[\Mbar_{0,k}(\Hb(\An),\beta)]^{\mathrm{red}}
\end{align*}

We have the following lemma, whose proof is nearly identical to
that of Lemma $2$ in \cite{okpanhilb}.
\begin{lem}\label{div}
The standard $T$-equivariant Gromov-Witten invariants of
$\Hb_{m}(\An)$ with nonzero degree and with insertions from
$H_{T}^{\ast}(\Hb_{m}(\An),\QQ)$ are divisible by $(t_1+t_2)$.
\end{lem}
\begin{proof}
As these invariants take values in $\QQ(t_1,t_2)$, divisibility is
defined by valuation respect to $(t_1+t_2)$.  The cohomology
$H_{T}^{\ast}(\Hb(\An),\QQ)$ is spanned as a $\QQ[t_1,t_2]$-module
by the Nakajima basis with respect to $(1, \omega_1, \dots,
\omega_n)$, so it suffices to prove this statement with insertions
from this basis.  If we instead study insertions given by the
Nakajima basis with respect to the fixed points $[p_{1}], \dots,
[p_{n+1}]$, then by compactness these invariants lie in
$\QQ[t_1,t_2]$ and by the reduced construction are divisible by
$(t_1+t_2)$.  Since the change of basis from $\{[p_{1}], \dots,
[p_{n+1}]\}$ to $\{1, \omega_{1}, \dots, \omega_{n}\}$ does not
introduce any denominators of $(t_1+t_2)$, we are done.
\end{proof}

We denote the reduced invariants with curved brackets; we have
shown that
$$\langle \overrightarrow{\mu}, \overrightarrow{\nu}\rangle^{\Hb}_{0,2,\beta} = (t_1+t_2)\left( \vec{\mu}, \vec{\nu} \right)^{\Hb}_{0,2,\beta}.$$


\subsection{Factorization}

The first application of reduced class arguments will be studying
two point invariants in the Nakajima basis with respect to $\{1,
\omega_{1},\dots,\omega_{n}\}$. The following result shows that we
can remove any parts labelled with $1$. It is the analog of the
additivity statements for $\Hb(\mathbb{C}^{2})$ proved in Section $3.5$
of \cite{okpanhilb}.

\begin{prop}\label{factorization}
We have the factorization
$$\langle \mu(1) \prod \lambda_{i}(\omega_{i})|\Theta_{+}| \nu(1) \prod \rho_{i}(\omega_{i})\rangle
= \langle \mu(1)|\nu(1)\rangle \cdot \langle \prod
\lambda_{i}(\omega_{i})|\Theta_{+}| \prod \rho_{i}(\omega_{i})\rangle.$$
\end{prop}
\begin{proof}
In the above expression, the first factor on the right-hand side
is just the usual inner product on Fock space, equivalent to the
classical Poincare pairing.  In particular, it is nonzero if and
only if $\mu = \nu$.

We can assume $l(\mu) \leq l(\nu)$.  We consider the associated
two-point invariant obtained by ordering the parts $\mu^{(k)}$ of
$\mu$ and labelling them with fixed points $[p_{i_{k}}]$:
$$\langle \mu^{(1)}([p_{i_{1}}])\dots \mu^{(l)}([p_{i_{l}}])\cdot \prod
\lambda_{i}(\omega_{i})|\Theta_{+}|\nu(1) \prod
\rho_{i}(\omega_{i})\rangle.$$

The (halved) cohomological degree of this invariant is 
$$2l(\mu) + m - l(\mu) + m - l(\nu) - (2m -1) = l(\mu) - l(\nu) + 1\leq 1.$$
Since the first insertion has compact support, concentrated along
the exceptional locus of $\An$, it forces the invariant to be a
polynomial in $t_1$ and $t_2$ that must also be divisible by
$(t_1+t_2)$ by Lemma \ref{div}. Therefore the invariant vanishes
unless $l(\mu) = l(\nu)$, in which case it equals
$\gamma\cdot(t_1+t_2)$ where
\begin{equation}\label{red inv}
 \gamma  = \left( \mu^{(1)}([p_{i_{1}}])\dots \mu^{(l)}([p_{i_{l}}])\cdot \prod
\lambda_{i}(\omega_{i}), \nu(1) \prod
\rho_{i}(\omega_{i})\right)^{\Hb}\in \QQ
\end{equation} is the reduced invariant.
Since $\gamma$ is a nonequivariant constant, it can be evaluated
by replacing the equivariant classes $[p_{i_{k}}]$ with
nonequivariant point classes $[\xi_{k}]$ for distinct points
$x_{i_{k}} \in \An$ that do not lie on any of the exceptional
divisors $E_{i}$.

The moduli space of maps connecting the Nakajima cycle
$\mu(\xi)\prod \lambda_{i}(\omega_{i})$ to the Nakajima cycle
$\nu(1)\prod \rho_{i}(\omega_{i})$ is empty unless $\mu = \nu$. In
this case, it can be identified with the space of maps to the
product
\begin{equation}\label{prod of punct}
\prod_{k}\Hb_{\mu_{k}}(\xi_{k}) \times Z\end{equation}
where $\Hb_{\mu_{k}}(\xi_{k})$ is the punctual Hilbert scheme of
$\mu_{k}$ points supported at the point $\xi_{k}$ and $Z$ is the
subscheme of $\Hb(\An)$ consisting of subschemes supported on the
exceptional locus.  This moduli space of maps to (\ref{prod of punct}) has connected components
corresponding to the contribution of each of these factors to the
total degree of the stable map.

The only component that contributes to our invariant is
where the map to each of the punctual Hilbert schemes is
contracted. Indeed, the standard obstruction theory factors into
contributions arising from each of the factors.  Moreover, if the
component of our curve class from any of these factors is nonzero
then the arguments from last section give rise to a trivial factor
to the obstruction theory.  If two of these components are
nonzero, then we have two trivial factors and the associated
reduced invariant must also vanish. Since we are considering 
non-punctual curve classes, we know that the component of our
curve class in $Z$ must be nonzero, and therefore the domain curve
is contracted in the projection to each factor
$\Hb_{\mu_{k}}(\xi_{k})$. The reduced invariant in $\eqref{red inv}$ is
given by
$$\prod_{k} \frac{1}{\mu_{k}} \cdot \left( 
\prod \lambda_{i}(\omega_{i}),  \prod
\rho_{i}(\omega_{i})\right)^{\Hb}.$$

Finally, if we substitute
$$1 = \sum_{k} \frac{[p_{k}]}{w^{i}_{L}\cdot w^{i}_{R}}$$
for the first insertion, the statement of the proposition follows
by a direct calculation.
\end{proof}

As a consequence of this proposition, if we proceed inductively on
the number of points $m$, it suffices to determine only a certain
minor of the full two-point matrix $\Theta_{+}$.  Moreover, from
this proof, these invariants are always of the form
$\gamma(t_1+t_2)$, where $\gamma$ is the nonequivariant reduced
invariant. This allows us to make the following useful
observation. In order to calculate the two-point operator $\Theta_{+}$
precisely, it is enough to calculate the above two-point invariants
$\mod (t_1+t_2)^{2}$.  If we then work with any other of the 
bases described in section \ref{bases}, it also suffices to determine
the two-point invariants with respect to the new basis $\mod
(t_1+t_2)^{2}$.  This is true (and well-defined) since the
coefficients in the change of bases never have factors of
$(t_1+t_2)$ in the denominator.

\section{Geometric calculations}\label{geometriccalculations}

In this section, we give an inductive procedure for determining
two-point invariants in terms of a fixed number of new
calculations for each $m$.

The strategy in this section is to apply virtual localization \cite{grabpan}
with respect to the torus action.
However, while the $T$-fixed points of $\Hb(\An)$ are isolated,
the loci of $T$-fixed curves are typically positive-dimensional
and quite complicated to describe concretely. Instead we proceed
indirectly and use the observation from last section to ignore
loci that contribute excess multiples of $(t_1+t_2)$.  This allows
us to reduce the analysis to correlators for which only
zero-dimensional loci of curves contribute.

\subsection{Fixed-point correlators}

Our procedure for determining these invariants proceeds via three
intermediate propositions about two-point invariants in the fixed-point
basis. Recall that we use the subscript $[i,j]$ to isolate the
contribution of curve classes on $\An$ that are linear
combinations of $E_{i}, \dots, E_{j-1}$ with nonzero coefficients
for $E_{i}$ and $E_{j-1}$.

\begin{prop}\label{fixedlemma1}
For $m \geq 1$, an arbitrary two-point correlator
$$\langle [J_{\vec{\lambda}}]| \Theta_{[i,j]}|[J_{\vec{\pi}}]\rangle \in
\QQ[t_1,t_2]((q))[[s_{i},\dots,s_{j-1}]]$$ is congruent modulo
$(t_1+t_2)^{2}$ to a linear combination of
\begin{itemize}
\item $\langle [J_{\vec{\eta}}]|\Theta_{[i,j]}| [J_{\vec{\eta}}]\rangle$ where $\eta_{k} =
\emptyset$ for $k \ne i$
\item two-point correlators $\langle [J_{\vec{\kappa}}]|\Theta_{[i,j]}| [J_{\vec{\sigma}}]\rangle$
for $\Hb_{m^{\prime}}(\An)$ with $m^{\prime} < m$
\end{itemize}
with coefficients in $R =\QQ(t_1,t_2)_{(t_1+t_2)}$
determined by pairings in the classical equivariant cohomology of
$\Hb(\An)$.
\end{prop}

For the next two propositions, we specify the following fixed points by
the multipartitions $\vec{\rho},
\vec{\theta},\vec{\kappa},\vec{\sigma}$:
\begin{gather*}
\rho_i=(m),\quad \rho_k=\emptyset, \mbox{ if } k\ne  i,\\
\theta_i=(1^m), \quad \theta_k=\emptyset, \mbox{ if } k\ne i,\\
\kappa_i=(m-1),\quad \kappa_j=(1),\quad\kappa_k=\emptyset,\mbox{
if } k\ne i,j,\\
\sigma_i=(1^{m-1}),\quad \sigma_j=(1),
\quad\sigma_k=\emptyset,\mbox{ if } k\ne i,j.
\end{gather*}
We then have the next two steps in our algorithm.
\begin{prop}\label{fixedlemma2}
For $m \geq 1$, a two-point correlator $\langle [J_{\vec{\eta}}]|\Theta_{[i,j]}|
[J_{\vec{\eta}}]\rangle$ where $\eta_{k} = \emptyset$ for
$ k \ne i$ is congruent modulo $(t_1+t_2)^{2}$ to a linear
combination of
\begin{itemize}
\item $\langle [J_{\vec{\rho}}]|\Theta_{[i,j]}|[J_{\vec{\kappa}}]\rangle$
\item $\langle [J_{\vec{\theta}}]|\Theta_{[i,j]}| [J_{\vec{\sigma}}]\rangle$
\item two-point correlators $\langle [J_{\vec{\lambda}}]|\Theta_{[i,j]}| J_{\vec{\eta}}]\rangle$
for $\Hb_{m^{\prime}}(\An)$ with $m^{\prime} < m$
\end{itemize}
with coefficients in $R$ determined by pairings in the classical equivariant cohomology of
$\Hb(\An)$.
\end{prop}

\begin{prop}\label{fixedlemma3}
For these special two-point correlators we have the following expression modulo 
$(t_1+t_2)^2$
\begin{gather*}
\langle
[J_{\vec{\rho}}]|\Theta_{[i,j]}|[J_{\vec{\kappa}}]\rangle=(-1)^{m-1}(t_1+t_2)
((n+1) t_1)^{2m}\frac{(m!)^2}{m}
\log(1-(-q)^{m-1}s_{ij}),\\
\langle
 [J_{\vec{\theta}}]|\Theta_{[i,j]}|[J_{\vec{\sigma}}]\rangle=
(-1)^{m-1}(t_1+t_2)((n+1) t_1)^{2m}\frac{(m!)^2}{m}
\log(1-(-q)^{-m+1}s_{ij}),
\end{gather*}
where $s_{ij}=s_i\cdot\dots \cdot s_{j-1}$ and the logarithm is expanded
with non-negative exponents in $s_{k}$.
\end{prop}

\subsection{Degree scaling}

As a first application of the above three steps, we easily show the
following proposition.
\begin{prop}\label{degreescaling}
The two-point operator $\Theta$ satisfies the following
properties:
\begin{itemize}
\item$\Theta_{+}(q,s_{1},\dots,s_{n}) = \sum_{1\leq i < j \leq n} \sum_{d \geq 1}
\Theta_{d\alpha_{i,j}}(q) (s_{i}\cdot\dots s_{j-1})^d$
\item $\Theta_{d\alpha_{i,j}}(-q) = \Theta_{\alpha_{i,j}}((-q)^{d})/d$
for $d \geq 1$.
\end{itemize}
\end{prop}

The first part of this proposition states that our answer only
depends on a sum over roots of the $A_n$ lattice.  The second part
states that taking multiples of a fixed root has an extremely
simple scaling property.  We will explain a second argument for the root dependence at the end of the paper.

\begin{proof}
We proceed by induction on $m$, the number of points, starting
from the vacuous case of $m=0$. Propositions \ref{fixedlemma1},\ref{fixedlemma2}, and \ref{fixedlemma3} then show that
two-point invariants in the fixed-point basis satisfy both claims
in the proposition modulo $(t_1+t_2)^{2}$.  This implies that the reduced
two-point invariants $\langle\prod \mu_{i}(\omega_{i}), \prod
\nu_{j}(\omega_{j})\rangle^{\Hb}$ also satisfy both claims modulo
$(t_1+t_2)^{2}$ and therefore precisely as well by the
observation at the end of last section.  The claim then follows
from the factorization of proposition~\ref{factorization}.
\end{proof}

\subsection{Localization and unbroken curves}

It suffices to evaluate the reduced two-point invariants
$$\left(  [J_{\vec{\lambda}}], [J_{\vec{\pi}}]\right)^{\Hb}_{[i,j]} \mod (t_1+t_2),$$
where again the subscript $[i,j]$ isolates curve classes with support in $[i,j]$.  While
the $T$-fixed loci are quite complicated to describe, many of these
loci will be contribute additional factors of $(t_1+t_2)$ and can
be ignored.  A description of which loci contribute nontrivially
and the proper framework for handling them is given in
\cite{okpanhilb}.  The discussion given here follows section $3.8$
of that paper.

Let $T^{\pm}$ denote the antidiagonal torus $\{(\xi, \xi^{-1}) \}
\subset T$. We will first analyze fixed loci with respect to the
smaller torus $T^{\pm}$. The fixed points on $\Hb_{m}(\An)$ for
$T^{\pm}$ are the same as those with respect to $T$ but the locus
of fixed maps is larger. Let $f \in
\Mbar_{0,2}(\Hb_{m}(\An),\beta)$ denote a $T^{\pm}$-fixed map $f:
C \rightarrow \Hb(\An)$.  We say that $f$ is {\it broken} if
either
\begin{itemize}
\item the domain $C$ contains a connected, $f$-contracted subcurve $C^{\prime}$ for which
the curve $C \backslash C^{\prime}$ has at least two connected
components which are not $f$-contracted.
\item Two non $f$-contracted components $P_{1},P_{2} \subset C$ meet at a node $s$ of $C$
and have tangent weights $w_{P_{1},s}, w_{P_{2},s}$ satisfy
$w_{P_{1},s} + w_{P_{2},s} \ne 0$.
\end{itemize}
If $f$ is broken, so is every map in the connected component of
the $T^{\pm}$-fixed locus containing $f$ and we can classify these
components as broken and unbroken.  We can decompose the
$T^{\pm}$-equivariant calculation for the reduced two-point
correlator into contributions from broken and unbroken fixed loci:
$$\left(  [J_{\vec{\lambda}}], [J_{\vec{\pi}}]\right)^{\Hb}_{[i,j]} =
\left(  [J_{\vec{\lambda}}],
[J_{\vec{\pi}}]\right)^{\mathrm{broken}}_{[i,j]} + \left(
[J_{\vec{\lambda}}],
[J_{\vec{\pi}}]\right)^{\mathrm{unbroken}}_{[i,j]}.$$ The
following lemma is proven in \cite{okpanhilb}.
\begin{lem}\label{vanish broken}
The $T^{\pm}$-equivariant broken contributions vanish.  In particular,
if there are no unbroken contributions, then $\left(  [J_{\vec{\lambda}}], [J_{\vec{\pi}}]\right)^{\Hb}_{[i,j]}$
vanishes $\mod t_1+t_2$.
\end{lem}

\subsection{Properties of unbroken curves}

Every nonempty unbroken connected component must contain a
$T$-fixed map, so we can determine necessary criteria for
two-point invariants to be nonzero $\mod (t_1+t_2)^{2}$.

We first give a description of nonpunctual one-dimensional
$T$-orbits on $\Hb_{m}(\An).$ In the following, given partitions
$\mu, \rho$ we denote by $\mu \cup \rho$ the partition obtained by
concatenating the lists of parts and $\mu^{\prime}$ the partition
obtained by taking the transpose of the associated Young diagram.
\begin{lem}\label{nonpunctualorbit}
Suppose we have a nonpunctual one-dimensional $T$-orbit $C$
connecting fixed points $J_{\vec{\lambda}}$ and $J_{\vec{\eta}}$.
Then there exists a unique $k$ and partitions $\mu, \nu, \rho$
such that
\begin{itemize}
\item $(1,\omega_{k})\cdot [C] > 0$,
\item $\lambda_{l} = \eta_{l}$ for $l \neq k, k+1$,
\item up to reordering the two fixed points we have
\begin{gather*}
\lambda_{k} = \mu \cup \rho,\quad \lambda_{k+1} = \nu^{\prime}\\
\eta_{k} = \mu,\quad \eta_{k+1} = (\nu \cup
\rho)^{\prime}\end{gather*}
\item and the tangent weight at $J_{\vec{\lambda}}$ is a positive integral multiple
of $w_{R}^{k}$.
\end{itemize}
If $l(\rho) = 1$ then there is a unique such $T$-orbit.
\end{lem}
\begin{proof}

The structure of a nonpunctual orbit near a fixed point 
can be analyzed using $\Hb(\mathbb{C}^{2})$ as a local model, in
which case, the third and fourth conditions are satisfied by
the analysis in section $7.2$ of \cite{nakajimabook}.  The transpose
occurs because of our convention for partition orientation.  
Since the tangent weights $w_{R}^{k}$ are not proportional to
any other tangent weights, this forces the uniqueness of $k$
as well as the first two conditions.  
\end{proof}

For punctual $T$-orbits we have the following lemma
\begin{lem}\label{punctualorbit}
Given a punctual $T$-orbit connecting $J_{\vec{\lambda}}$ and
$J_{\vec{\eta}}$, then there exists a unique $k$ such that
\begin{itemize}
\item
$\lambda_{l} = \eta_{l}$ for $l \ne k$,
\item
the tangent weight at $J_{\vec{\lambda}}$ is of the form
$$a\cdot w_{L}^{k} + b\cdot w_{R}^{k}$$
with $a\cdot b \leq 0$.
\end{itemize}
\end{lem}
\begin{proof}
At a fixed points of $\Hb(\mathbb{C}^{2})$, the tangent weights are of
the form $at_{1} +bt_{2}$ with $a\cdot b \leq 0$.  One can check
directly that the associated linear combinations $a\cdot w_{L}^{k}
+ b\cdot w_{R}^{k}$ are distinct as we vary $k$. Therefore any
$T$-orbit is noncontracted over a unique fixed point and the
statement follows.
\end{proof}

Let $f: (C,z_{1},z_{2}) \rightarrow \Hb_{m}(\An)$ be a
two-pointed, $T$-fixed unbroken map. Any noncontracted irreducible
component of $C$ has at most $2$ marked or nodal points.  This
implies that $C$ must be a chain of rational curves
$$C = C_{1} \cup C_{2} \cup \dots \cup C_{r}$$
with nodes  $q_{1}, \dots, q_{r-1}$. Up to relabelling, we must
have
$$z_{1} \in C_{1}, z_{2} \in C_{r}$$
or
$$z_{1},z_{2} \in C_{1}$$
and $C_{1}$ is $f$-contracted.

In the first case, let $S$ denote the sequence of $T$-fixed points
of $\Hb(\An)$
$$S = f(z_{1}), f(q_{1}), f(q_{2}), \dots, f(q_{r}), f(z_{2});$$
we otherwise take $S$ to be
$$S = f(q_{1}), f(q_{2}), \dots, f(q_{r}).$$

We define an ordering  on fixed points given as follows. Given a
partition $\lambda$ with parts $\lambda^{(1)}, \dots,
\lambda^{(l)}$, we have the function
$$\epsilon(\lambda) = \sum \binom{\lambda^{(i)}}{2}.$$
Given two multipartitions $\vec{\lambda},
\vec{\eta}$, we say that
$$\vec{\lambda} \succeq \vec{\eta}$$
if, when $k$ is the smallest integer such that $|\lambda_{k}| \ne
|\eta_{k}|$, we have $|\lambda_{k}| > |\eta_{k}|$ or, if $|\lambda_{k}|
= |\eta_{k}|$ for all $k$, then $\epsilon(\lambda_{k}) \geq
\epsilon(\eta_{k})$ for all $k$.

\begin{lem}\label{partialordering}
The sequence $S$ is either an increasing sequence or a decreasing
sequence with respect to the ordering $\succeq$.
\end{lem}
\begin{proof}
This follows from the description of the one-dimensional orbits.
We discuss the case where $z_{1}$ and $z_{2}$ are on noncontracted
components.  Assume that the (fractional) tangent weight to
$C_{1}$ at $z_{1}$ is congruent to $a t_{1} \mod (t_1+t_2)$ with
$a< 0$. Because of the unbroken condition, the sum of the tangent
weights at each node $q_{k}$ must be proportional to
$t_{1}+t_{2}$, so the tangent weight to $C_{k}$ at $q_{k-1}$ is $a
t_{1}\mod(t_1+t_2)$ for all $k$.

For a nonpunctual component $C_{k}$, it is a multiple cover of one
of the orbits described in Lemma~\ref{nonpunctualorbit}.  Since $a < 0$, the
tangent space at $q_{k-1}$ maps to a $T$-fixed point with tangent
weight a positive rational multiple $w_{R}^{l}$ for some $l$,
since $w_{L}^{l} = bt_{1} \mod (t_1+t_2)$ with $b > 0$.  The length of the
support of the subscheme at $p_{l}$ decreases and increases at
$p_{l+1}$, so we have $f(q_{k-1}) \succeq f(q_{k})$.

For a punctual component $C_{k}$ connecting partitions $\lambda$
and $\eta$ supported at the fixed point $p_{l}$, the tangent
weight at $q_{k-1}$ is a positive rational multiple of $a w_{L}$
degree with respect to $D$ is given by the localization expression
$0 < D \cdot C_{k} = \frac{c(\lambda) - c(\eta)}{a w_{L}^{l} +
bw_{R}^{l}}$ where $c(\lambda)$ is the content
$$c(\lambda) = D|_{J_{\lambda}} = \sum_{(i,j) \in \lambda} (i-1)w_{L}^{l}
+(j-1)w_{R}^{l}.$$ The coefficient of $w_{L}^{l}$ in $c(\lambda)$
is precisely $\epsilon(\lambda)$, which implies $\epsilon(\lambda)
\geq \epsilon(\eta)$, so again we have $f(q_{k-1}) \succeq
f(q_{k})$.
\end{proof}

The main application of this unbroken analysis is the following
vanishing proposition for two-point fixed correlators.
\begin{prop}\label{vanishing}
Given two distinct $(n+1)$-tuples of partitions $\vec{\lambda}
\neq \vec{\eta}$ such that either $|\lambda_{i}| = |\eta_{i}|$ or
$|\lambda_{j}| = |\eta_{j}|$ then
$$
\langle  [J_{\vec{\lambda}}]|\Theta_{[i,j]}| [J_{\vec{\eta}}]\rangle = 0
\mod (t_1+t_2)^{2}.
$$
\end{prop}
\begin{proof}
Since we are considering curve classes $\beta$ with support equal
to $[i,j]$, we have $(1,\omega_{k})\cdot\beta=0$ for $k < i$ and
$(1,\omega_{i})\cdot \beta > 0$. Suppose there exists an unbroken
map $f$; since $\vec{\lambda} \neq \vec{\eta}$, the two marked
points must be at opposite ends of the chain $C$.  The first
condition and the description of punctual and nonpunctual
$T$-orbits ensures that $|\lambda_{k}| = |\eta_{k}|$ for $k < i$.
Similarly, the second condition implies that the length of the
subscheme supported at $p_{i}$ either increases or decreases at
some point in the chain.  Lemma \ref{partialordering} implies that
$|\lambda_{i}| \neq |\eta_{i}|$ and the analogous statement for
$j$.  This contradicts the hypothesis, so all $T$-fixed maps are
broken.
\end{proof}

\subsection{Proof of Proposition \ref{fixedlemma1}}

The proof of this proposition and the next will only use the
factorization and vanishing statements of Propositions \ref{factorization}
and \ref{vanishing}.  

Recall that our invariants take values in $R = \QQ[t_1,t_2]_{(t_{1}+t_{2})}$.We define an equivalence relation $\sim$ on elements of
$R((q))[[s_{1},\dots,s_{n}]]$ so that
$$f(t_1,t_2,q,s_1,\dots,s_n) \sim 0$$
if $f \mod (t_1+t_2)^{2}$ is an $R$-linear combination of
 $\langle [J_{\vec{\kappa}}]| \Theta_{i,j}|[J_{\vec{\sigma}}]\rangle$
for $\Hb_{m^{\prime}}(\An)$ with $m^{\prime} < m$.

In these terms, the factorization proposition can be restated as
follows.
\begin{lem}\label{fact2}
Given $k > 0, \vec{\mu} \in H^{\ast}(\Hb_{m}(\An),\QQ), \vec{\nu}
\in H^{\ast}(\Hb_{m-k}(\An),\QQ)$, we have
$$\langle\vec{\mu} | \Theta_{i,j} |  \pp_{-k}(1)\vec{\nu}\rangle \sim 0.$$
\end{lem}
\begin{proof}
If we write $\vec{\mu},\vec{\nu}$ in terms of the Nakajima basis
with respect to $1, \omega_{1},\dots, \omega_{n}$, this follows
immediately from the statement of proposition~\ref{factorization} since $k > 0$.
\end{proof}

In the case of $\mathcal{A}_{1}$, we will restate this lemma in
terms of the symmetric function notation for elements of
$\mathcal{F}_{\mathcal{A}_{1}}$. Since
$$\pp_{-k}(1) = \frac{1}{2t_1(t_2-t_1)}\pp_{-k}([p_{1}]) +
\frac{1}{2t_2(t_1-t_2)}\pp_{-k}([p_{2}])$$ and each of these
fixed-point Nakajima operators corresponds to multiplication by
$\mathsf{p}_{k}(z)$, we have
$$\langle (\mathsf{p}_{\mu}(z)\cdot g(z)) \otimes h(z) | \Theta_{+}| \bullet\rangle \sim (-1)^{l(\mu)}
\langle g(z) \otimes (\mathsf{p}_{\mu}(z)\cdot h(z))| \Theta_{+}|
\bullet\rangle,$$ where $\bullet$ denotes any cohomology class.

If we apply Lemma \ref{fact2} to invariants with insertions in the
Nakajima basis with respect to the basis $\{1, p_{1}, \dots,
\widehat{p_{j}}, p_{n}\}$, this shows that all two-point
correlators are determined by two-point correlators
$$ \langle [J_{\vec{\lambda}}]|\Theta_{i,j}| [J_{\vec{\eta}}]\rangle $$
with $\lambda_{j}=\emptyset, \eta_{j} = \emptyset$. By proposition \ref{vanishing},
 we must have $\vec{\lambda} = \vec{\eta}$.  For
$\mathcal{A}_1$, this concludes the proof.

For the case of general $\An$, we have to argue further as
follows. Given two partitions $\pi$ and $\pi^{\prime}$ of size $a$
and $a-1$ respectively, we say that
 $$\pi \searrow \pi^{\prime}$$ if their
 Young diagrams differ by the removal of a single box.
We already have that $\lambda_{j} = \emptyset$; suppose further
that there exists $k \ne i,j$ for which $\lambda_{k} \ne
\emptyset$.  Then there exists a multipartition
$\vec{\eta}$ such that
$$\lambda_{r} = \eta_{r}, \quad r\ne k$$
and
$$\lambda_{k} \searrow \eta_{k}.$$

If we use symmetric function notation for the fixed-point basis,
then we have that modulo $(t_1+t_2)$
$$\pp_{-1}(1)\left( \bigotimes_{r} \mathsf{J}_{\eta_{r}}(z^{(r)})\right) 
\equiv (n+1)t_{1} \sum_{s=1}^{n+1}
\bigotimes_{r\ne s} \mathsf{J}_{\eta_{r}}(z^{(r)}) \otimes \left((\mathsf{p}_{1}(z^{(s)})\cdot
\mathsf{J}_{\eta_{s}}(z^{(s)})\right).$$ 
Moreover, if we expand out
these products in the basis of Jack polynomials, we have that
$$\mathsf{p}_{1}(z)\cdot \mathsf{J}_{\eta_{k}}(z) = c\cdot \mathsf{J}_{\lambda_{k}}(z) + \dots$$
with $c \not\equiv 0 \mod t_{1}+t_{2}$. This implies that
\begin{align*}
0 &\sim \langle \bigotimes \mathsf{J}_{\lambda_{r}}(z^{(r)}) | \Theta_{i,j} |
\pp_{-1}(1)\cdot \bigotimes
\mathsf{J}_{\eta_{r}}(z^{(r)})\rangle \\
&\sim\sum_{s=1}^{n+1} \langle \mathsf{J}_{\lambda_{r}}(z^{(r)})|\Theta_{i,j}|
\bigotimes_{r\ne s}
\mathsf{J}_{\eta_{r}}(z^{(r)})\otimes (\mathsf{p}_{1}(z)\cdot \mathsf{J}_{\eta_{s}}(z))\rangle\\
&\sim c \langle \bigotimes \mathsf{J}_{\lambda_{r}}(z)| \Theta_{i,j}|
\bigotimes \mathsf{J}_{\lambda_{r}}(z)\rangle.
\end{align*}
The last equality follows from Proposition \ref{vanishing}, since the
number of points supported at either $p_{i}$ or $p_{j}$ is fixed.

\subsection{Proof of Proposition \ref{fixedlemma2}}

We will just prove this for $\mathcal{A}_{1}$, since the case of
general $n$ is essentially the same.  We will write everything
using symmetric function notation for fixed-point basis elements.  Moreover, since we are ignoring
factors of $t_{1}+t_{2}$, it is convenient to work with Schur
polynomials instead of Jack polynomials, since the branching rules
are easier to describe. In these terms, the goal of this section
is to calculate
$$\langle \mathsf{s}_{\lambda}(z)\otimes 1|\Theta_{1,2}| \mathsf{s}_{\lambda}(z)\otimes 1\rangle$$
in terms of
$$\langle \mathsf{s}_{\rho}(z)\otimes 1| \Theta_{1,2}| \mathsf{s}_{\kappa}(z)\otimes \mathsf{s}_{(1)}(z) \rangle$$
and
$$\langle \mathsf{s}_{\theta}(z)\otimes 1|\Theta_{1,2}|\mathsf{s}_{\sigma}(z)\otimes \mathsf{s}_{(1)}(z)\rangle$$
where $\rho = (m), \kappa = (m-1), \theta = (1^{m}), \sigma =
(1^{m-1})$.

We first assume that there are two distinct partitions,
$\lambda_{1}, \lambda_{2}$ of size $m-1$ such that
$$\lambda \searrow \lambda_{1}, \lambda \searrow \lambda_{2}$$
This happens if and only if $\lambda \ne (a, a, \dots, a)$ for
some $a$ dividing $m$.

We then have the equality
\begin{align*}
0 &\sim \langle \mathsf{s}_{\lambda_{1}}(z)\otimes \mathsf{s}_{(1)}(z) | \Theta_{1,2}|
\mathsf{s}_{\lambda_{2}}(z)\otimes
\mathsf{s}_{(1)}(z)\rangle\\
 & \sim \langle (\mathsf{s}_{(1)}(z)\cdot \mathsf{s}_{\lambda_{1}}(z))\otimes 1 | \Theta_{1,2}| 
 (\mathsf{s}_{(1)}(z)\cdot \mathsf{s}_{\lambda_{2}}(z))\otimes 1\rangle\\
 & \sim \sum_{\eta_{1}\searrow \lambda_{1}} \sum_{\eta_{2}\searrow \lambda_{2}}
 \langle \mathsf{s}_{\eta_{1}}(z)\otimes 1 |\Theta_{1,2}| \mathsf{s}_{\eta_{2}}(z)\otimes 1\rangle\\
 & \sim \langle \mathsf{s}_{\lambda}(z)\otimes 1 | \Theta_{1,2}| \mathsf{s}_{\lambda}(z)\otimes 1\rangle.
 \end{align*}
 The first equality follows from Proposition \ref{vanishing}, since $\lambda_{1}\ne
\lambda_{2}$.
 The second equality is the restatement of the factorization lemma in terms of symmetric
functions given after lemma \ref{fact2}.  The third equality is 
the branching rule formula for multiplying
 Schur functions.  The fourth equality follows again from the vanishing statement and the fact
 that $\lambda$ is the unique partition $\eta$ of $m$ such that
 $$\eta\searrow \lambda_{1}, \quad \eta\searrow \lambda_{2}.$$

This reduces us to the case where $\lambda = (a, a, \dots, a)$
for some $a | m$ or, equivalently, where the Young diagram of
$\lambda$ is rectangular.  If $a \ne 1, m$, there are two distinct
ways of removing two boxes from the Young diagram of $\lambda$:
$$\lambda_{1}=(a, a, \dots, a-1, a-1), \quad \lambda_{2}=(a, a, \dots, a, a-2).$$
We then again have
\begin{align*}
0 &\sim \langle \mathsf{s}_{\lambda_{1}}(z)\otimes \mathsf{s}_{(2)}(z)|\Theta_{1,2}|
\mathsf{s}_{\lambda_{2}}(z)\otimes
\mathsf{s}_{(1,1)}(z) \rangle \\
&\sim \langle \mathsf{s}_{(1,1)}(z)\cdot \mathsf{s}_{\lambda_{1}}(z)\otimes 1\
|\Theta_{1,2}| \mathsf{s}_{(2)}(z)\cdot
\mathsf{s}_{\lambda_{2}}(z)\otimes 1\rangle \\
& \sim \langle \mathsf{s}_{\lambda}(z)\otimes
1|\Theta_{1,2}|\mathsf{s}_{\lambda}(z)\otimes 1\rangle
\end{align*}
again using Proposition \ref{vanishing} and the branching rule for Schur
functions.

This leaves $\lambda = \rho$ or $\lambda = \theta$.  In these
cases, we have for instance
$$\langle \mathsf{s}_{\theta}(z)\otimes 1|\Theta_{1,2}|\mathsf{s}_{\sigma}(z)\otimes \mathsf{s}_{1}(z)\rangle \sim (-1)\cdot
\langle \mathsf{s}_{\theta}(z)\otimes 1|\Theta_{1,2}|\mathsf{s}_{\theta}(z)\otimes
1\rangle.$$ We argue similarly in the case of $\lambda = \rho.$

\subsection{Proof of Proposition \ref{fixedlemma3}}

In order to prove propositon \ref{fixedlemma3}, we will explicitly
calculate the two invariants

$$\langle [J_{\vec{\theta}}]|\Theta_{i,j}| [J_{\vec{\sigma}}]\rangle,
\langle [J_{\vec{\rho}}]|\Theta_{i,j}| [J_{\vec{\kappa}}]\rangle \mod
(t_1+t_2)^{2}.$$

The main observation in each case is that, for every curve class,
there is at most one unbroken $T$-fixed curve joining the two fixed
point insertions in each of these invariants.   If we assume this
for now, then since every positive-dimensional compact variety with a
$T$-action has at least two fixed points, this curve is the entire
unbroken $T^{\pm}$-fixed locus. By Lemma~\ref{vanish broken}, we have
\begin{align*}
\left( [J_{\vec{\theta}}], [J_{\vec{\sigma}}]\right)^{T} \mod 
(t_1+t_2) &\equiv \left(
[J_{\vec{\theta}}], [J_{\vec{\sigma}}]\right)^{T^{\pm}}\\
&= \left( [J_{\vec{\theta}}], [J_{\vec{\sigma}}]\right)^{T^{\pm},
\mathrm{unbroken}}
\end{align*}
where curved brackets denote the reduced virtual class invariants.
These invariants differ from the usual invariants by a factor of
$(t_{1}+t_{2})$.  Therefore, as long as we work $\mod
(t_{1}+t_{2})^{2}$, it suffices to calculate the localization
residue of the unique unbroken $T$-fixed curve.

We begin by explaining the calculation of the first invariant.

\begin{lem}
The only unbroken maps joining $J_{\vec{\theta}}$ to
$J_{\vec{\sigma}}$ have multidegree given by the monomial
$(q^{1-m}s_{i}\cdot\dots\cdot s_{j-1})^{d}$ for $d \geq 1$.
Moreover, for each $d$, there is a unique such map.
\end{lem}
\begin{proof}
We first show there is a unique chain of $T$-orbits that is
decreasing with respect to the partial ordering $\succeq$.  Let
$f: C \rightarrow \Hb(\An)$ be an unbroken $T$-fixed map with
$$C = C_{1} \cup C_{2} \cup \dots \cup C_{r}$$
and let $S = \{J_{\theta} , f(q_{1}), \dots,f(q_{r-1}),
J_{\sigma}\}$ be the decreasing sequence of fixed points associated
to the marked points and nodes. For $1\leq k \leq j-i$, let
$\theta^{(k)}$ be the multipartition defined by
$$\theta^{(k)}_{i} = (1^{m-1}), \theta^{(k)}_{i+k} = (1), \theta^{(k)}_{l} = \emptyset,\quad l
\ne i, i+k.$$ We claim that
$$S = \{ \theta, \theta^{(1)}, \theta^{(2)},\dots, \theta^{(j-1)} = \sigma\}.$$

Indeed, let $L_{1}, \dots L_{r}$ denote the $T$-orbits that are
the reduced images of $C_{k}$. First, $L_{1}$ must be a
nonpunctual $T$-orbit with a single point moving from $p_{i}$ to
$p_{i+1}$, so $f(q_{1}) = \theta^{(1)}$.  This is due to the fact
that the function $\epsilon$ achieves a strict minimum on
$(1^{m})$.  For this same reason, all subsequent $L_{k}$ are
nonpunctual curves as well which forces $f(q_{k}) = \theta^{(k)}.$
The uniqueness of these orbits follows from the full statement of
Lemma \ref{nonpunctualorbit}.

The degrees of each of these orbits is calculated as follows
\begin{gather*}
D \cdot L_{1} = \frac{-\binom{m}{2}w_{R}^{i} +
\binom{m-1}{2}w_{R}^{i}}{w_{R}^{i}} = 1-m
\\
(1,\omega_{i})\cdot L_{1} = 1, (1,\omega_{l})\cdot L_{1} = 0
\end{gather*}
Similarly, we have
$$D\cdot L_{k} = 0, (1,\omega_{k})\cdot L_{k} = 1.$$

Suppose that the degree of $C_{1}$ over $L_{1}$ is $d$.  The
tangent weights at $f(q_{k})$ are $w_{L}^{i+k}$ and $w_{R}^{i+k}$,
which add to $t_{1}+t_2$. Since the fractional tangent weights to
$C$ at each node $q_{k}$ must add to a multiple of $t_1+t_2$, we
have that the degree of $C_{k}$ over $L_{k}$ is $d$ for all $k$.
Therefore, the only unbroken fixed maps have degree given by
$q^{d(1-m)}\cdot(s_{i}\cdot\dots s_{j-1})^{d}$ and, for each $d$,
there is a unique such map $f$.
\end{proof}

This calculation of the residue at this unbroken map $f$ factors
into contributions from each component $C_{k}$, each node $q_{k}$,
and each marked point at $J_{\vec{\theta}}$ and
$J_{\vec{\sigma}}$. For more details on how to determine these
contributions, we refer the reader to \cite{grabpan}.

\begin{enumerate}
\item[$\bullet$]Contribution from $C_{1}$:

This contribution is given by the ratio of equivariant Euler
classes
$$
\frac{1}{d}\cdot \frac{e(H^{1}(C_{1},
f^{\ast}(T_{\Hb(\An)})))}{e(H^{0}(C_{1},
f^{\ast}(T_{\Hb(\An)}))-0)}.
$$
The $\frac{1}{d}$ factor arises from the automorphism of $C_{1}$
over $L_{1}$. The $H^{0}$-term has a single trivial weight $0$
corresponding to reparameterization that we remove by hand in the
above expression.  The restriction of the tangent bundle to the
orbit $L_{1}$ is given by
$$T_{\Hb(\An)}\mid_{L_{1}} = \mathcal{O}(2)\oplus \mathcal{O}(-2)\oplus
\mathcal{O}(1)^{m-1} \oplus \mathcal{O}(-1)^{m-1}.$$ The $T$
weights over each fixed point of these bundles are given by the
following table.
\begin{center}
  \begin{tabular}{c|c|c}
 & $J_{\vec{\theta}}$ & $J_{\vec{\sigma}}$ \\
\hline
$\mathcal{O}(2)$ & $w_{R}^{i}$ & $-w_{R}^{i}$ \\
$\mathcal{O}(-2)$ & $w_{L}^{i}$ & $w_{L}^{i}+2w_{R}^{i}$ \\
$\mathcal{O}(1)$ & $2 w_{R}^{i}$ & $w_{R}^{i}$ \\
$\vdots$ & $\vdots$ & $\vdots$ \\
$\mathcal{O}(1)$ & $m w_{R}^{i}$ & $(m-1)w_{R}^{i}$\\
$\mathcal{O}(-1)$ & $w_{L}^{i} - w_{R}^{i}$ & $w_{L}^{i}$ \\
$\vdots$ & $\vdots$ & $\vdots$ \\
$\mathcal{O}(-1)$ & $w_{L}^{i}-(m-1)w_{R}^{i}$ &
$w_{L}^{i}-(m-2)w_{R}^{i}$
  \end{tabular}
\end{center}
In what follows, let $\tau \equiv (n+1)t_{1} \mod (t_1+t_2)$. The
contribution of $H^{0}(C_{1},f^{\ast}(\mathcal{O}(2))- 0$ is given
by
$$\prod_{k=1}^{d} -(\frac{k}{d} w_{R}^{i})^{2} \equiv
(-1)^{d}\left(\frac{d!}{d^d}\right)^{2} \tau^{2d}.$$ The
contribution of $H^{0}(C_{1},f^{\ast}(\mathcal{O}(1)))$ for $a
=1,\dots, m-1$ is
$$(-1)^{d+1}\tau^{d+1} \prod_{k=0}^{d} (a+\frac{k}{d})\mod(t_{1}+t_{2}).$$
The contribution of $H^{1}(C_{1},f^{\ast}(\mathcal{O}(-2)))$ is
given by
$$(t_{1}+t_{2})\cdot
(-1)^{d-1}\left(\frac{d!}{d^d}\right)^{2}\tau^{2d-2}\mod(t_{1}+t_{2})^{2}.$$
The contribution of $H^{1}(C_{1},f^{\ast}(\mathcal{O}(-1)))$ for
$a =1,\dots, m-1$ is
$$\tau^{d-1} \prod_{k=1}^{d-1} (a+\frac{k}{d})\mod(t_{1}+t_{2}).$$

The total contribution is then
$$\frac{(-1)^{md+m+d}}{d(m-1)!m!} \frac{1}{\tau^{2m}}\cdot (t_{1}+t_{2}) \mod
(t_{1}+t_{2})^{2}.$$

\item[$\bullet$]Contribution from $C_{k}, k > 1$:

The same calculation shows that the contribution here is given by
$$\frac{(-1)^{m}}{d\cdot((m-1)!)^{2}}\cdot (t_{1}+t_{2}) \mod (t_{1}+t_{2})^{2}.$$

\item[$\bullet$]Contribution from nodes:

At each node $q_{k}$ we have a contribution
$$ \frac{d}{t_{1}+t_{2}}$$
arising the normal direction to the fixed locus arising from
smoothing the node, and
$$ (-1)^{m} \tau^{2m}((m-1)!)^{2}\mod(t_{1}+t_{2})$$
from the gluing condition at the node.

\item[$\bullet$]Contribution from marked points:

The contribution from each marked points is the product of the
tangent weights at the fixed point:
$$(-1)^{m}\tau^{2m}(m!)^{2}, (-1)^{m}\tau^{2m}((m-1)!)^{2} \mod (t_{1}+t_{2}).$$

\end{enumerate}

The total contribution yields
\begin{align*}
\langle [J_{\vec{\theta}}]|\Theta_{i,j}|[J_{\vec{\sigma}}]\rangle
&=\sum_{d \geq 1}
\frac{1}{d}\frac{(m!)^2}{m}(-1)^{md+m+d}\tau^{2m} (q^{1-m}s_{i}\dots s_{j-1})^{d}\\
&= (-1)^{m-1}((n+1)t_{1})^{2m}\frac{(m!)^{2}}{m} \log (1 -
(-q)^{1-m}s_{i}\cdot\dots\cdot s_{j-1}).
\end{align*}

For the calculation of
$$\langle [J_{\vec{\rho}}]|\Theta_{i,j}|[J_{\vec{\kappa}}]\rangle,$$
we argue in the same way, omitting the details since they are
similar to the last calculation.
 As before, there is a unique sequence of $T$-orbits
connecting the two fixed points, that is decreasing with respect
to the partial ordering $\succeq$. Let $\rho^{(k)}$ be the
multipartitions defined by
\begin{gather*}
\rho^{(1)}_{i}= (m-1,1), \rho^{(1)}_{l} = \emptyset,\quad l \ne i \\
\rho^{(k)}_{i} = (m-1), \rho^{(k)}_{i+k-1} = (1), \rho^{(k)}_{l} =
\emptyset.
\end{gather*}
Then the sequence of marked points and nodes associated to any
fixed, unbroken map is
$$S = \{\vec{\rho}, \rho^{(1)}, \dots, \rho^{(j-i+1)} = \vec{\kappa}\}.$$
The argument for this is analogous to the argument given above,
now using the fact that the two largest values of the function
$\epsilon$ are achieved at the partitions
$$(m), (m-1,1).$$

Again, the $T$-orbits $L_{1}, \dots, L_{j-i+1}$ joining the fixed
points in $S$ are uniquely determined.  The orbit $L_{1}$ is
punctual with tangent weights $\pm ((m-1)w_{L}^{i} - w_{R}^{i})$
and has degree given by the monomial $q$.  The orbit $L_{2}$ has
degree given by the monomial $q^{-1}s_{i}$ while the remaining
$L_{k}$ are nonpunctual with tangent weights $\pm \tau \mod
(t_{1}+t_{2})$ and degree $s_{k}, i< k < j$. Given an unbroken
$T$-fixed map $f: C \rightarrow \Hb(\An)$, $C = C_{1}\cup\dots\cup
C_{r}$ as before, if the degree of $C_{2}$ over $L_{2}$ is $d$,
then the unbroken condition forces the degree of $C_{1}$ over
$L_{1}$ to be $md$ and the degree of $C_{k}$ over $L_{k}$ for $ k
\geq 2$ is $d$.  Therefore there is again a unique unbroken map
for every $d \geq 1$.

\begin{enumerate}

\item[$\bullet$] Contribution for $L_{1}$:

This punctual contribution is the main calculation of
\cite{okpanhilb}:
$$ (t_{1}+t_{2})\cdot (-1)^{md+m}\frac{(m!)^{2}}{md}\tau^{2m} \mod(t_{1}+t_{2})^{2}.$$

\item[$\bullet$] Contribution for the node $q_{1}$:

The contribution from smoothing the node at this point yields
$$\frac{md}{(m-1)(t_{1}+t_{2})}.$$

\item[$\bullet$] Contribution for $L_{2}$:

This calculation here involves an enumeration of the
$T$-equivariant splitting of the normal bundle, similar to the one
given above.
$$\frac{1}{d} (t_1+t_2)\cdot
\frac{(-1)^{m+d}}{((m-2)!)^{2}(m-1)m}\frac{1}{\tau^{2m}}
\mod(t_{1}+t_{2})^{2}.$$

\item[$\bullet$] Other curves, nodes, and marked points:
The remaining calculations are the same as those given above:
$$(-1)^{m}((m-1)!)^{2}\tau^{2m}.$$
\end{enumerate}

The total contribution yields
\begin{align*}
\langle [J_{\vec{\rho}}]|\Theta_{i,j}|[J_{\vec{\kappa}}]\rangle
&=\sum_{d \geq 1}
\frac{1}{d}\frac{(m!)^2}{m}(-1)^{md+m+d}\tau^{2m} (q^{m-1}s_{i}\dots s_{j-1})^{d}\\
&= (-1)^{m-1}((n+1)t_{1})^{2m}\frac{(m!)^{2}}{m} \log (1 -
(-q)^{m-1}s_{i}\cdot\dots\cdot s_{j-1}).
\end{align*}

\section{Operator calculations}\label{operatorcalculations}

In this section, we prove the results of the last section for the
operator $\Omega_{+}(q, s_{1},\dots,s_{n})$.  The results of last
section give an inductive algorithm (Propositions \ref{fixedlemma1}-\ref{fixedlemma3}) which uniquely determines the
operator $\Theta_{+}$ in terms of
\begin{enumerate}
\item the factorization statement of Proposition \ref{factorization},
\item the vanishing statement of Proposition \ref{vanishing}
\item the calculation of two-point correlators $\mod (t_1+t_2)^{2}$ in Proposition
\ref{fixedlemma3}.
\item the statement that the coefficients of $\langle \prod \mu_{i}(\omega_{i}|\Theta_{+}|\prod \nu_{i}(\omega_{i})\rangle$ are linear polynomials, proportional to $(t_1+t_2)$
\item and the vanishing vacuum expectation $\langle v_{\emptyset}|\Theta_{+}|v_{\emptyset}\rangle$.
\end{enumerate}
All the other manipulations of the last section involved moving
between the Nakajima and fixed-point bases.  If we prove each of
these statements for $(t_{1}+t_{2})\Omega_{+}$, then this inductive algorithm
forces
$$(t_{1}+t_{2})\Omega_{+}= \Theta_{+}.$$

It is clear from its definition that $\Omega_{+}$ has the same
root dependence and degree scaling properties as those of
$\Theta_{+}$ proved in Proposition \ref{degreescaling}.  It therefore suffices
to isolate the contribution of a fixed root $\alpha =
\alpha_{i,j}$ or, equivalently, to study the coefficient of
$s_{i}\cdot \dots \cdot s_{j-1}:$
$$\mathcal{E}_{\alpha}(q) = -\sum_{k\in \ZZ} :e_{j,i}(k)e_{i,j}(-k):(-q)^{k},$$
where we have dropped the factor of $t_{1}+t_{2}$ for convenience.
The vanishing vacuum expectation is immediate from this formula.

\subsection{Commutation relations}

We first write down the commutation relations for
$\mathcal{E}_{\alpha}(q)$ with the Nakajima operators
$\pp_{r}(\gamma)$ for $r \ne 0$.  Define operators
$$\mathcal{E}^{r}_{\alpha}(q) = -\sum_{k \in \ZZ}
e_{j,i}(k)e_{i,j}(r-k)(-q)^{k-\frac{r}{2}}.$$ Notice that, on any
fixed graded piece of Fock space, the summands in the above
expression vanish for $k$ sufficiently negative. For $r=0$, the
discrepancy between this operator and $\mathcal{E}_{\alpha}$
arises from the normal ordering.  This can be written as
$$\mathcal{E}^{0}_{\alpha}(q)= \mathcal{E}_{\alpha}(q) +
(e_{ii}(0)-e_{jj}(0))\frac{-q}{1+q} +\frac{q}{(1+q)^{2}}\cdot
c.$$ In particular, we have
$$\mathcal{E}^{0}_{\alpha}(q)v_{\emptyset} = \frac{q}{(1+q)^{2}} v_{\emptyset}.$$
It follows directly from the embedding of the Heisenberg algebra
into $\gotg$ and from the commutation relations that
\begin{gather*}
\left[ \mathfrak{p}_{r}(\gamma),
\mathcal{E}^{s}_{\alpha}(q)\right] = (\alpha, \gamma)((-q)^{-r/2}-(-q)^{r/2})\cdot
\mathcal{E}_\alpha^{r+s}(q),
\end{gather*}
where $(\alpha,\gamma)$ denotes the Poincare pairing on
$H^{\ast}_{T}(\An,\QQ)$.

This result makes calculating matrix elements with respect to the
Nakajima basis extremely simple. In particular, we can easily
deduce the factorization statement.
\begin{prop}\label{omegafact}
In terms of the Nakajima basis with respect to
$\{1,\omega_{1},\dots,\omega_{n}\}$, we have
$$\langle \mu(1) \prod\lambda_{i}(\omega_{i})|\mathcal{E}_{\alpha}(q)| \nu(1)
\prod\rho_{i}(\omega_{i})\rangle = \langle \mu(1)|\nu(1)\rangle \cdot
\langle \lambda_{i}(\omega_{i})|\mathcal{E}_{\alpha}(q)|
\rho_{i}(\omega_{i})\rangle.$$
\end{prop}
\begin{proof}
This follows immediately from the fact that
$$[\mathcal{E}_{\alpha}(q), \mathfrak{p}_{r}(1)] = 0$$
which in turn follows from the vanishing
$$(\alpha, 1) = 0.$$
\end{proof}

Furthermore, it is clear that matrix elements 
$\langle \prod \lambda_i(\omega_i)|\mathcal{E}_{\alpha}(q)|\prod \rho_{i}(\omega_{i})\rangle$
are nonequivariant constants, so that the coefficients of $(t_1+t_2)\Omega_{+}$ are linear.  All that remains are the vanishing statement and the exact evaluations.

\subsection{Diagonalization}\label{diagonalization}

We now analyze the matrix elements of
$\mathcal{E}^{0}_{\alpha}(q)$ with respect to the fixed-point
basis.  For our purposes, we only need to understand a certain
minor of this matrix $\mod (t_{1}+t_{2})$.  More precisely, since all operators and bases are defined over
$R = \QQ[t_1,t_2]_{(t_{1}+t_{2})}$, 
it makes sense to study the reduction of these operators
after tensoring with 
$$R/(t_{1}+t_{2}) = \QQ(\tau),$$
where $\tau \equiv (n+1)t_{1}\mod t_{1}+t_{2}$.

Recall the tensor product
decomposition
$$\mathcal{F}_{\An}\otimes R = \bigotimes_{k=1}^{n+1} \mathcal{F}_{\mathbb{C}^{2},k}\otimes R,$$
where the factors correspond to the $T$-fixed points $p_{1},\dots,
p_{n+1}$ of $\An$. Since
$$(\alpha, [p_{k}]) = 0$$
for $k \ne i,j$, we have
$$[\mathcal{E}^{0}_{\alpha}, \mathfrak{p}_{r}([p_{k}])] = 0$$
for $k \ne i,j$.  Therefore, this operator admits a decomposition
$$\mathcal{E}^{0}_{\alpha}(q) = \mathcal{E}^{\prime}_{\alpha}(q)\otimes \mathrm{Id}$$
where $\mathcal{E}^{\prime}_{\alpha}(q)$ is an operator on
$$\mathcal{F}_{i,j} \otimes R= (\mathcal{F}_{\mathbb{C}^{2},i}
\otimes\mathcal{F}_{\mathbb{C}^{2},j})\otimes R$$ and the identity matrix
acts on the remaining factors. From now on, we suppress the
$\mathbb{C}^{2}$ in our notation. There is a bi-grading on the above
space given by $$\mathcal{F}_{i,j} = \bigoplus_{a,b\geq 0}
\mathcal{F}^{(a)}_{i}\otimes \mathcal{F}^{(b)}_{j}.$$ We are
interested in the minors with respect to this decomposition.  That
is, if $\pi_{a,b}$ denote the orthogonal projections with respect
to this grading, then
$$\mathsf{M}_{\alpha}(q) = \bigoplus_{a,b\geq 0} \pi_{a,b}\circ
\mathcal{E}^{\prime}_{\alpha}(q)\circ \pi_{a,b}$$ is the operator
defined by restricting to the minors where the number of points
concentrated at each fixed point is held constant.  We want to
diagonalize the operator
$$\overline{\mathsf{M}_{\alpha}}(q)$$ on
$$\mathcal{F}_{i,j}\otimes R/(t_{1}+t_{2}) = \mathcal{F}_{i,j}\otimes \QQ(\tau)$$
obtained by extension of scalars.

The advantage of composing with the projectors $\pi_{a,b}$ is that
we have a further factorization of
$\overline{\mathsf{M}_{\alpha}}(q)$. It will again be convenient
to write everything in terms of symmetric function notation for
elements of $\mathcal{F}= \mathcal{F}_{\mathbb{C}^{2}}\otimes \QQ(\tau)$.


Let $\mathsf{A}(q)$ be the operator on symmetric functions defined
by
\begin{align*}
\mathsf{A}(q)\cdot \mathsf{p}_{\mu}(z) &= \sum_{\nu} \mathsf{p}_{\nu}(z)\cdot \\
&\left(\sum_{\rho \subset \mu\cap \nu}
\frac{1}{\mathfrak{z}(\nu\backslash\rho)} \cdot
\frac{(-q)^{1/2}}{(1+q)}\cdot
f_{\mu\backslash\rho}(q)f_{\nu\backslash\rho}(q)\right).
\end{align*}
In the above expression, $f_{\lambda}(q)$ is defined by
$$f_{\lambda}(q) = \prod_{i}\left((-q)^{\lambda_{i}/2} - (-q)^{-\lambda_{i}/2}\right).$$

If $\omega$ is the involution on symmetric functions defined by
$$\omega(\mathsf{p}_{\mu}(z)) = (-1)^{l(\mu)}\mathsf{p}_{\mu}(z),$$
then we also have the conjugate operator
$$\mathsf{B}(q) = \omega\circ \mathsf{A}(q) \circ \omega.$$
We then have the decomposition:
\begin{lem}
$$\overline{\mathsf{M}_{\alpha}}(q) = -\mathsf{B}(q)\otimes\mathsf{A}(q).$$
\end{lem}
\begin{proof}
This follows from the formulas for $\mathsf{A}(q)$ and
$\mathsf{B}(q)$ and the commutation relations for
$\overline{\mathsf{M}_{\alpha}}(q)$ with respect to the Nakajima
operators.
\end{proof}

The following lemma is proven in \cite{okpanp1}.
 In what follows, given a partition $\lambda$, we define the function
$$e(\lambda,q) = \sum_{i} (-q)^{\lambda_{i}-i+\frac{1}{2}},$$
which is a rational function in $(-q)^{1/2}$.
\begin{lem}
The operators $\mathsf{A}(q)$ and $\mathsf{B}(q)$ are
diagonalizable with eigenvectors given by Schur polynomials
$\mathsf{s}_{\lambda}(z)$ and eigenvalues $e(\lambda,q)$ and
$e(\lambda^{\prime},q)$ respectively.
\end{lem}
\begin{proof}
In \cite{okpanp1}, the Gromov-Witten theory of $\mathbf{P}^1$ is
given in terms of an operator
$\mathcal{E}^{0}_{\mathbf{P}^{1}}(z)$ on the space of symmetric
functions (identified with standard Fock space). Its eigenvectors
are given by Schur functions $\mathsf{s}_{\lambda}(z)$ with eigenvalues
$e({\lambda},-e^{z})$. Moreover, a direct comparison of the
commutation relations with the Nakajima operators shows that
$$\mathsf{A}(e^{z}) = \mathcal{E}^{0}_{\mathbf{P}^{1}}(z).$$
The statement for $\mathsf{B}(q)$ follows from the fact that
$$\omega(s_{\lambda}) = s_{\lambda^{\prime}}.$$
\end{proof}

The following proposition, including a vanishing statement, is an immediate corollary of the above
discussion.

\begin{prop}\label{omegavan}
The matrix $\overline{\mathsf{M}_{\alpha}}(q)$ is diagonal with
respect to the basis
$$\mathsf{J}_{\lambda}(z)\otimes \mathsf{J}_{\rho}(z)$$
with eigenvalues
$$-e_{\lambda^{\prime}}(q)\cdot e_{\rho}(q).$$
Furthermore, given two distinct multipartitions
$\vec{\lambda} \ne \vec{\eta}$ such that either $|\lambda_{i}| =
|\eta_{i}|$ or $|\lambda_{j}| = |\eta_{j}|$ then
$$
\langle  [J_{\vec{\lambda}}]|\mathcal{E}_{\alpha}(q)|
[J_{\vec{\eta}}]\rangle = 0 \mod (t_1+t_2).
$$
\end{prop}

\subsection{Two-point correlators}

We now prove a version of Proposition \ref{fixedlemma3} for
$\Omega_{+}(q)$.  Let
Let $\Omega_{[i,j]}(q,s_{i},\dots,s_{j-1})$ denote the contribution
to $\Omega_{+}$ arising from monomials in $s_{i},\dots, s_{j-1}$ (or, equivalently,
curve classes supported in $[i,j]$).
Recall the multipartitions
$\vec{\rho}, \vec{\kappa}, \vec{\theta},\vec{\sigma}$ from the statement 
of proposition \ref{fixedlemma2}.

\begin{prop}\label{omegaexact}
We have the following evaluations for two-point correlators modulo $(t_1+t_2)$:
\begin{gather*}
\langle
[J_{\vec{\rho}}]|\Omega_{[i,j]}|[J_{\vec{\kappa}}]\rangle=(-1)^{m-1}
((n+1) t_1)^{2m}\frac{(m!)^2}{m}
\log(1-(-q)^{m-1}s_{ij}),\\
\langle
 [J_{\vec{\theta}}]|\Omega_{[i,j]}|[J_{\vec{\sigma}}]\rangle=
(-1)^{m-1}((n+1) t_1)^{2m}\frac{(m!)^2}{m}
\log(1-(-q)^{-m+1}s_{ij}).
\end{gather*}
\end{prop}
\begin{proof}

This immediately reduces to a claim regarding
$\overline{\mathcal{E}_{\alpha}}(q)$.  
Using the commutation relations with Nakajima operators, it
suffices to restrict to the case of $n=1$. If we write everything
in terms of symmetric functions, the first equality is equivalent
to
$$\langle \mathsf{s}_{(m)}(z)\otimes 1| \overline{\mathcal{E}_{\alpha}(q)}| \mathsf{s}_{(m-1)}(z)\otimes
\mathsf{s}_{1}(z)\rangle = (-1)^{m}(-q)^{m-1}.$$

This identity then follows from the eigenvalue calculations of the
last section:
\begin{multline*}
\langle \mathsf{s}_{(m)}(z)\otimes 1|\overline{\mathcal{E}_{\alpha}(q)}|
\mathsf{s}_{(m-1)}(z)\otimes \mathsf{s}_{1}(z)\rangle = \langle \mathsf{s}_{(m)}(z)\otimes 1|
\overline{\mathcal{E}_{\alpha}(q)}|\mathfrak{p}_{-1}(1)
\mathsf{s}_{(m-1)}(z)\otimes 1\rangle - \\ \langle \mathsf{s}_{(m)}(z)\otimes
1|\overline{\mathcal{E}_{\alpha}(q)}| (\mathsf{s}_{(m-1)}(z)\cdot
\mathsf{s}_{1}(z))\otimes 1\rangle
= \langle \mathfrak{p}_{-1}^{\ast}(1)\mathsf{s}_{(m)}(z)\otimes 1|
\overline{\mathcal{E}_{\alpha}(q)}| \mathsf{s}_{(m-1)}(z)\otimes 1\rangle -\\
\langle \mathsf{s}_{(m)}(z)\otimes
1|\overline{\mathcal{E}_{\alpha}(q)}|(\mathsf{s}_{(m)}(z)+\mathsf{s}_{(m-1,1)}(z))\otimes 1\rangle
=  \langle \mathsf{s}_{(m-1)}(z)\otimes 1| \overline{\mathcal{E}_{\alpha}(q)}|
\mathsf{s}_{(m-1)}(z)\otimes 1\rangle + \\
 \langle \mathsf{s}_{(m)}(z)\otimes 1|
\overline{\mathcal{E}_{\alpha}(q)}| \mathsf{s}_{(m)}(z)
\otimes 1\rangle
= (-1)^{m-1}e(\emptyset,q)(e((m-1),q) - e((m),q)) =\\
-(-1)^{m-1}(-q)^{m-1}.
\end{multline*}
For the second equality, we replace the partition $(m)$ with
$(1^{m})$ which is equivalent to replacing $q$ with $1/q$.
\end{proof}

\section{Proofs of Main Results}

\subsection{Punctual Contribution}

In this section, we explain how to identify the punctual
contributions of the operators $\Theta$ and $\Omega$, that is to show
$(t_1+t_2) \Omega_{0}(q) = \Theta_{0}(q)$ This
contribution is expressed in terms of the two-point operator for
$\Hb(\mathbb{C}^{2})$.  In fact, Kiem and Li \cite{kiemli} have used the case of $\mathbb{C}^2$ to calculate 
the punctual two-point operator for an arbitrary surface $S$ using
a more general version of the reduced class arguments we make here.
However, since we need an equivariant statement, we argue directly.

We first rewrite $\Omega_{0}(q)$ in terms of  the
identification $\mathcal{F}_{\An}\otimes \QQ(t_1,t_2) = \bigotimes
\mathcal{F}_{\mathbb{C}^{2}}\otimes \QQ(t_1,t_2)$. It follows from
definition that
\begin{align}\label{omega0}
(t_1+t_2)\Omega_{0}(q) &= -(t_1+t_2)\sum_{i=1}^{n+1} \sum_{k\geq 1}
\frac{1}{w_{L}^{i}w_{R}^{i}}\mathfrak{p}_{-k}([p_{i}])\mathfrak{p}_{k}([p_{i}])
\log\left(\frac{1-(-q)^k}{1-(-q)}\right)\\
&= \sum_{i=1}^{n+1} \Theta_{\mathbf{C}^{2},i}(q) \nonumber
\end{align}
where the last equality denotes the two-point operator for
$\Hb(\mathbb{C}^{2})$ acting on the $i$-th factor of the tensor product
decomposition, as calculated in \cite{okpanhilb}. Our goal is to
prove the same statement for $\Theta_{0}(q)$.

The first step is to apply the reduced class construction to prove
an analog of the factorization statement.

\begin{prop}\label{punctual1}
We have the following two identities for insertions labelled with
$1$.

\begin{multline*}
\langle \mu(1) \prod \lambda_{i}(\omega_{i})| \Theta_{0}| \nu(1)
\prod \rho_{i}(\omega_{i})\rangle = \langle
\mu(1)|\Theta_{0}|\nu(1)\rangle \cdot \langle \prod
\lambda_{i}(\omega_{i})| \prod\rho_{i}(\omega_{i})\rangle +\\
\langle \mu(1)|\nu(1)\rangle \cdot \langle \prod
\lambda_{i}(\omega_{i})|
\Theta_{0}(q)|\prod\rho_{i}(\omega_{i})\rangle
\end{multline*}
\begin{equation*}
 \langle \mu(1)|
\Theta_{0}(q) | \nu(1)\rangle = (t_1+t_2)\langle \mu(1) | \Omega_{0}(q) |
\nu(1)\rangle
\end{equation*}
\end{prop}

\begin{proof}

For the first claim, we proceed exactly as in the proof of
Proposition \ref{factorization}.  If $l(\mu) \geq l(\nu)$, we
expand the cohomology class $1$ in terms of fixed points $[p_{i}]$.
The same dimension analysis forces $l(\mu) = l(\nu)$ and in fact
$\mu = \nu$.  Moreover, we again rewrite the invariant (up to a factor of $(t_1+t_2)$) as a
nonequivariant reduced invariant
$$\left( \mu(1) \prod \lambda_{i}(\omega_{i}), \nu([p]) \prod
\rho_{i}(\omega_{i})\right)^{\mathrm{red}}$$ where $p$ is a point
on $\An$ away from the exceptional locus.  Again, if the curve
class is not contracted over both $p$ and the exceptional locus,
there is a second trivial factor to the obstruction theory that
forces vanishing.  The only difference is that both possibilities
can now occur, since we have a punctual curve class. They give the
two contributions to the right-hand side.

For the second claim, we use the expression for $\Omega_{0}(q)$
from equation \eqref{omega0}.  Again, it suffices to prove the
claim after expanding $1$ in terms of fixed points in the
insertions associated to $\nu$.  The same nonequivariant-reduced argument shows that
$$
\frac{\langle \mu(1)| \Theta_{0}(q)| \prod \nu_{i}([p_{i}])
\rangle}{\langle \mu(1) | \prod \nu_{i}([p_{i}])\rangle} =
\sum_{i=1}^{n+1} \frac{\langle \mu(1)| \Theta_{0}(q)|
\nu_{i}([p_{i}]) \rangle}{\langle \mu(1) |
\nu_{i}([p_{i}])\rangle}.
$$
The analogous equality is trivially true for $\Omega_{0}(q)$ using the
tensor-product decomposition of \eqref{omega0}.  This reduces the
claim to the invariant $\langle \mu(1)| \Theta_{0}(q)|\mu([p_{i}])\rangle$ for fixed $i$.  Since
everything is concentrated at a single fixed point, we can
directly identify the invariant with a calculation on $\mathbb{C}^{2}$ to
give
$$\langle \mu(1)| \Theta_{0}(q)| \mu([p_{i}])\rangle = \langle \mu(1)|
\Theta_{\mathbf{C}^{2},i}(q)|\mu([p_{i}])\rangle = (t_1+t_2)\langle \mu(1)|
\Omega_{0}(q)| \mu([p_{i}])\rangle.$$

\end{proof}

\begin{prop}\label{punctual2}
\begin{multline*}\langle \mu(1) \prod \lambda_{i}(\omega_{i})|
\Omega_{0}| \nu(1) \prod \rho_{i}(\omega_{i})\rangle = \langle
\mu(1)|\Omega_{0}|\nu(1)\rangle \cdot \langle \prod
\lambda_{i}(\omega_{i})| \prod\rho_{i}(\omega_{i})\rangle +\\
\langle \mu(1)|\nu(1)\rangle \cdot \langle \prod
\lambda_{i}(\omega_{i})|
\Omega_{0}(q)|\prod\rho_{i}(\omega_{i})\rangle
\end{multline*}
\end{prop}
\begin{proof}
This follows immediately from the fact that the operators
$\mathfrak{p}_{k}(1)$ and $\mathfrak{p}_{k}(\omega_{i})$ commute.
\end{proof}

The main proposition now follows quite easily.

\begin{prop}\label{punctualtheorem}
$$\Theta_{0}(q) = (t_1+t_2)\Omega_{0}(q).$$
\end{prop}
\begin{proof}
We proceed by induction on the number of points $m$. We first
observe that if either $\overrightarrow{\mu}$ or
$\overrightarrow{\nu}$ contains a part labelled by $1$, then 
Propositions \ref{punctual1} and \ref{punctual2} along with the inductive hypothesis
establish the equality.

We work in the Nakajima basis with respect to the fixed points
$p_{i}$.  Given multipartitions $\overrightarrow{\mu}$,
$\overrightarrow{\nu}$, there is a trivial vanishing
$$\langle \prod \mu_{i}([p_{i}])| \Theta_{0}(q) | \prod \nu_{i}([p_{i}])\rangle = \langle
\prod \mu_{i}([p_{i}])| \Omega_{0}(q) | \prod
\nu_{i}([p_{i}])\rangle = 0$$ if there exists an $i$ for which
$$|\mu_{i}| \ne |\nu_{i}|.$$
For $\Omega_{0}(q)$ this follows from the tensor-product
decomposition; for $\Theta(q)$, this is because the number of
points concentrated at each fixed point must be constant for
punctual curves. In the general case, if we assume $\mu_{1} \ne
\emptyset$, we expand
$$[p_{1}] = c_{0}\cdot 1 + c_{2}[p_{2}]+\dots +c_{n+1}[p_{n+1}]$$
for the insertions of $\mu_{1}$.  The terms corresponding to
$p_{2}, \dots, p_{n+1}$ cannot contribute by the same trivial
vanishing statement. We thus have
$$\langle \prod \mu_{i}([p_{i}])| \Theta_{0}(q) | \prod \nu_{i}([p_{i}])\rangle = \langle
\mu_{1}(c_{0}\cdot 1) \prod \mu_{i}([p_{i}])| \Theta_{0}(q) |
\prod \nu_{i}([p_{i}])\rangle$$ and the analogous equality for
$\Omega_{0}(q)$.  The claim then follows from our initial
observation.
\end{proof}

\subsection{Proof of Theorem \ref{theorem1}}

For the nonpunctual contributions, the equality $$\Theta_{+} =
(t_1+t_2)\Omega_{+}$$ follows from the arguments of section \ref{geometriccalculations} and
\ref{operatorcalculations}. More precisely, we have shown in these sections that both
operators satisfy factorization and vanishing statements. These are propositions \ref{factorization} and \ref{vanishing} for $\Theta_{+}$ and propositions \ref{omegafact} and \ref{omegavan} for $\Omega_{+}$.
The algorithm in section \ref{geometriccalculations} explains how these propositions
reduce the calculation of each operator to the precise
calculations of propositions \ref{fixedlemma3} and \ref{omegaexact}, where equality is
proven directly.  For the punctual contribution, the equality
$$\Theta_{0} = (t_1+t_2)\Omega_{0}$$
is proposition \ref{punctualtheorem}.

\subsection{Generation conjecture}\label{gener}

Recall corollary \ref{div mult} gives the operators for quantum
multiplication by divisors $D$ and $(1,\omega_{i})$:
\begin{gather*}
M_{D} = M_{D}^{cl} + (t_1+t_2) q \frac{d}{dq}\Omega(q,s_1,\dots,s_n)\\
M_{(1,\omega_{i})} = M_{(1,\omega_{i})}^{cl} +
(t_1+t_2)s_{i}\frac{d}{ds_{i}} \Omega_{+}(q,s_1,\dots,s_n).
\end{gather*}
Since quantum cohomology defines a graded commutative ring, these
operators commute.  It therefore makes sense to discuss the joint
spectrum of these operators.  We conjecture the following
nondegeneracy statement for this spectrum.

\begin{conjecture}
The joint eigenspaces for the operators $M_{D},
M_{(1,\omega_{i})}$ are one-dimensional for all $m>0$.
\end{conjecture}

Under the assumption of this conjecture, we can immediately
generate the full quantum ring as well as all genus $0$
invariants.
\begin{corstar}
Assuming the above conjecture, the divisors $D$ and
$(1,\omega_{i})$ generate the small quantum cohomology ring
$H_{T}^{\ast}(\An, \QQ)$.  Moreover, we can calculate an arbitrary
$k$-point genus $0$ Gromov-Witten invariant for $\Hb(\An)$ in
terms of these operators.
\end{corstar}
\begin{proof}
The quantum cohomology ring
$H_{T}^{\ast}(\An,\QQ)\otimes\QQ(q,s_{1},\dots,s_{n})$ is
semisimple since it is a deformation of the semisimple classical
equivariant cohomology. The idempotents are simultaneous
eigenvectors for $M_{D}, M_{(1,\omega_{i})}$.  Given the nondegeneracy conjecture, a Vandermonde
argument shows that the vectors
$$M_{D}^{a}\prod_{i} M_{(1,\omega_{i})}^{b_{i}} (1^{m}),\quad a, b_{i}\geq 0$$
span the full cohomology ring.  The second statement follows from
standard reconstruction statements for varieties whose small
quantum cohomology ring is generated by divisors
(\cite{kon-manin,okpanhilb}).
\end{proof}

While we are unable to prove our nondegeneracy conjecture, we can
show the following partial result.

\begin{prop}
The operator for the purely quantum contribution
$$M_{D}(q,s_{1},\dots,s_{n}) - M_{D}^{cl} =
(t_{1}+t_{2})q\frac{d}{dq}\Omega(q,s_{1},\dots,s_{n})$$ has
distinct eigenvalues.
\end{prop}
\begin{proof}
We will sketch the proof for $\mathcal{A}_{1}$; the same argument
works in general. In order to show that $q\frac{d}{dq}\Omega(q,s)$
has distinct eigenvalues, we write our operator as a power series
in $s$ and apply a perturbation theory argument.  That is, we
write
$$q\frac{d}{dq}\Omega(q,s) = q\frac{d}{dq}\Omega_{0}(q) + s q\frac{d}{dq}
\mathcal{E}(q)+O(s^{2}),$$ and study the restriction of
$q\frac{d}{dq}\mathcal{E}(q)$ to the degenerate eigenspaces of the
punctual operator $\Omega_{0}(q)$. From the explicit form of
$\Omega_{0}(q)$, we know that an eigenbasis is given by the
Nakajima basis with fixed-point insertions; for every partition
$\mu$ of $m$, we have a degenerate eigenspace
$$V_{\mu} = \mathrm{Span}\{ \mu_{1}([p_{1}])\mu_{2}([p_{2}]) | \mu_{1}+\mu_{2} = \mu\}.$$
It suffices to show that the restriction and projection of
$q\frac{d}{dq}\mathcal{E}(q)$ to each $V_{\mu}$ has distinct
eigenvalues \cite{kato}.

Let $r_{k}$ denote the multiplicity of $k$ as a part of $\mu$.
Then the eigenspace $V_{\mu}$ admits a tensor product
decomposition.
$$V_{\mu} = \otimes_{k} S^{r_{k}}(V_{(k)}).$$
Moreover, the first-order perturbation also admits a factorization
$$\frac{(1+q)^{2}}{q}\mathcal{E}_{0}(q)|_{V_{\mu}} = \otimes S^{r_{k}}
\left(\frac{(1+q)^{2}}{q}\mathcal{E}_{0}(q)|_{V_{(k)}}\right).$$ For
each fixed part $k$, the eigenvectors of
$\frac{(1+q)^{2}}{q}\mathcal{E}_{0}(q)|_{V_{(k)}}$ are easily seen
to be
$$\frac{1}{2t_{1}}k([p_{1}]) - \frac{1}{2t_{2}}k([p_{2}]),
\frac{1}{t_{2}-t_{1}}k([p_{1}]) + \frac{1}{t_{1}-t_{2}}k([p_{2}])
$$ with eigenvalues given by
$$1, 1+ \frac{2}{k}((-q)^{k/2}-(-q)^{-k/2})^{2}$$
respectively.  The eigenvectors and eigenvalues in the general
case are obtained by multiplying these eigenvalues over the
distinct parts of $\mu$.  The eigenvalues are given by
$$\prod_{k} \left( 1+ \frac{2}{k}((-q)^{k/2}-(-q)^{-k/2})^{2}\right)^{s_{k}},\quad 0\leq s_{k}
< r_{k}$$ and are clearly distinct for different values of
$s_{k}$.  Finally, the eigenvalues for the actual perturbation are
obtained by taking the derivative of the above expression and are
again clearly distinct.
\end{proof}

We consider this proposition good evidence for the generation
conjecture since we are taking an operator-valued function of $q$
and $s$ with nondegenerate spectrum and  adding an operator
with no $q,s$-dependence.  Unfortunately, this fact alone is not sufficient to
prove the claim. In the case of $\mathbb{C}^{2}$ where the purely
quantum part of the operator again has nondegenerate spectrum, it
is possible to find a limit in the equivariant parameters in which
the classical contribution is suppressed.  However, a similar
argument does not seem possible in our case.

\subsection{Comparison with Gromov-Witten theory of $\An\times\mathbf{P}^{1}$}

We follow here notation from section $1.3$ and section $4.1$ of
\cite{gwan}.  The generating function
$$\mathsf{Z}^{\prime}_{GW}(\An\times\mathbf{P}^1)_{\overrightarrow{\mu}, \overrightarrow{\nu},
\overrightarrow{\rho}}$$ encodes the all-genus Gromov-Witten
theory of $\An\times\mathbf{P}^{1}$ with relative conditions given
by the cohomology-weighted partitions $\vec{\mu}, \vec{\nu},
\vec{\rho}.$ In \cite{gwan}, it is explained how to calculate
these invariants when one of these relative conditions corresponds
to a divisor insertion.  By comparing with our calculations here,
we have the following proposition.
\begin{prop}
Under the variable substitution $q = -e^{iu}$, we have
\begin{align*}
(-1)^{m}\langle\overrightarrow{\mu},(2),\overrightarrow{\nu}\rangle^{\mathrm{Hilb}}_{\An}
=
(-iu)^{-1+l(\mu)+l(\nu)}\mathsf{Z}^{\prime}(\An\times\mathbf{P}^1)_{\overrightarrow{\mu},(2),\overrightarrow{\nu}}
\end{align*}
and
\begin{align*}
(-1)^{m}\langle\overrightarrow{\mu},(1,\omega_{k}),\overrightarrow{\nu}\rangle^{\mathrm{Hilb}}_{\An}
=
(-iu)^{l(\mu)+l(\nu)}\mathsf{Z}^{\prime}(\An\times\mathbf{P}^1)_{\overrightarrow{\mu},(1,\omega_{k}),\overrightarrow{\nu}}.
\end{align*}
\end{prop}
\begin{proof}
Note that the substitution makes sense because our three-point
functions are rational functions of $q$. In \cite{gwan}, an analog
of the root dependence and degree scaling statement are proven; it
therefore suffices to consider $\mathcal{A}_{1}$ and the
coefficient of $s^{1}$.    In that case, we can compare
$\Omega_{+}(q,s)$ with the expression $\Theta^{\bullet}(u,s)$ from that
paper.
\end{proof}

There is a direct geometric relationship between the relative
Gromov-Witten theory of $\mathcal{A}_{1}\times\mathbf{P}^{1}$ and
the stationary theory of $\mathbf{P}^{1}$. This explains why the
operator controlling this latter theory played a role in section
\ref{diagonalization}.

We can use the all-genus relative Gromov-Witten theory of $\An
\times \mathbf{P}^{1}$ to define a new ring deformation of
$H_{T}^{\ast}(\Hb(\An),\QQ)$ over the ring
$$R^{\mathrm{GW}}=\mathbb{Q}(t_1,t_2)((u))[[ s_1,\dots, s_n]].$$

Given three cohomology-weighted partitions $\overrightarrow{\mu},
\overrightarrow{\nu}, \overrightarrow{\rho}$ of $m$, we define a
product $\ast$ using the following structure constants
$$\langle \overrightarrow{\mu}, \overrightarrow{\nu} \ast
\overrightarrow{\rho} \rangle =
(-iu)^{-m+l(\mu)+l(\nu)+l(\rho)}\mathsf{Z}^{\prime}(\An\times\mathbf{P}^1)_{\overrightarrow{\mu},
\overrightarrow{\nu}, \overrightarrow{\rho}.}$$ The following
proposition is established in \cite{gwan}.

\begin{prop}
The $R^{\mathrm{GW}}$-module $H_{T}^{\ast}(\Hb(\An),\QQ)\otimes R^{\mathrm{GW}}$ with the
product defined above satisfies the axioms of an $R^{\mathrm{GW}}$-algebra.
\end{prop}

In terms of this ring structure, corollary~\ref{div mult} states that we
can identify the multiplication operators for $(2)$ and $(1,
\omega_{i})$. Under the assumption of the generation conjecture,
this would imply the following general statement of the
Gromov-Witten/Hilbert correspondence. 

\begin{corstar}
Under the variable substitution $q = -e^{iu}$, we have
\begin{align*}
\langle\overrightarrow{\mu},\overrightarrow{\nu},\overrightarrow{\rho}\rangle^{\mathrm{Hilb}}_{\An}
=
(-iu)^{-m+l(\mu)+l(\nu)+l(\rho)}\mathsf{Z}^{\prime}(\An\times\mathbf{P}^1)_{\overrightarrow{\mu},\overrightarrow{\nu},\overrightarrow{\rho}}.
\end{align*}
\end{corstar}

In this form, this statement generalizes the correspondence proven
for $\mathbb{C}^{2}$ in \cite{localcurves},\cite{okpanhilb}.
Heuristically, it states that the two ring deformations defined by
quantum cohomology and the relative theory of
$\An\times\mathbf{P}^{1}$ are isomorphic after a specific change
of variables.

Both ring structures defined in this correspondence make sense for
an arbitrary surface $S$.  However, the precise relation given
here does not hold in this generality.  For instance, in the case
of $m=1$, the quantum cohomology structure constants have no
$q$-dependence while the Gromov-Witten theory of
$S\times\mathbf{P}^{1}$ in general will have nontrivial
$u$-dependence. This is already in the case of $\mathbf{P}^{2}$.
The special feature of the geometry that allows us to make such a
simple matching is the existence of a holomorphic symplectic form.

Notice that in the ring isomorphism, no change of variables is
required on the parameters corresponding to curve classes on the
surfaces $\An$; similarly, the Nakajima basis is identified
directly with relative conditions in the Gromov-Witten theory of
$\An\times\mathbf{P}^{1}$.  If we weaken this strong constraint on
the change of variables, it is reasonable to expect the
correspondence to generalize.


\section{Further directions}

\subsection{Quantum differential equation}
The quantum differential equation associated to $\Hb_{m}(\An)$ is the system
of linear differential equations given by
\begin{equation}\label{Qdiffeq}
\begin{cases}
&q\frac{\partial}{\partial q}\psi=M_{D}\psi\\
&s_i\frac{\partial}{\partial s_i}\psi=M_{(1,\omega_i)}\psi
\end{cases}
\quad \psi(q,s_{i})\in H_T^*(Hilb_m(\An))((q))[[s_{i}]],
\end{equation}
This system defines a flat connection $\nabla_{\An}(t_{1},t_{2})$ on $\mathbb{C}^{n+1}$ for the trivial bundle associated to $H^*(\Hb_{m}(\An),\CC)$
with regular singularities along the hypersurfaces
\begin{align}
(-q)^k&=1,\quad k=1,\dots,m-1,\label{first sing}\\
(s_{i}\cdot\dots\cdot s_{j-1})(-q)^{k}&=1,\quad k=-m+1,\dots,m-1,\quad 1\leq i < j \leq n+1.
\end{align}

In the case of $\mathbb{C}^{2}$, the quantum differential equation can be viewed
as a nonstationary extension of the Calogero-Sutherland system with 
extremely well-behaved monodromy properties \cite{okpanhilb}.
If we take the differential equation obtained by tensoring $n+1$ copies of this nonstationary system,
\eqref{Qdiffeq} can be viewed as a deformation of this system in the $s$-variables 
where the coupling between the different factors is essentially controlled by our operator
$\Omega_{+}$.  Moreover, as we discuss next, many of the qualitatively nice
features of the monodromy of $\nabla_{\mathbb{C}^{2}}$ extend to this deformation.

As an example, in the case of $n=1$ with $m=2$ points, the associated matrices are given as 
follows.
We order the standard basis of $H_T(\Hb_2(\AAA_1))$ in the following
way:

$$\frac{\mathfrak{p}^2_{-1}(E)}2 v_\emptyset,\quad
\frac{\mathfrak{p}_{-2}(E)}2v_\emptyset,\quad
\mathfrak{p}_{-1}(E)\mathfrak{p}_{-1}(1) v_\emptyset,\quad
\frac{\mathfrak{p}_{-1}^2(1)}2 v_\emptyset,\quad
\frac{\mathfrak{p}_{-2}(1)}2 v_\emptyset.$$

The notation $\theta=t_1+t_2$ is used in the formulas for the
divisor multiplication operators:

\begin{gather*}M_{(1,\omega)}=\left(\begin{matrix}
\frac{2\theta s(q+1/q+2s)}{(1+sq)(1+s/q)}&
\frac{\theta s(1/q-q)}{(1+sq)(1+s/q)}&-1&0&0\\
\frac{2\theta s(1/q-q)}{(1+s/q)(1+s/q)}&
\frac{\theta q(1+1/q)^2(1+s)}{(1+sq)(1+s/q)(1-s)}&0&0&-1\\
2t_1t_2&0&\frac{\theta(1+s)}{s-1}&-\frac{1}2&0\\
0&0&t_1t_2&2\theta&0\\
0&2t_1t_2&0&0&2\theta
\end{matrix}\right),\\
M_{D}=\left(\begin{matrix}
\frac{2\theta(q+1/q+2s)}{(1+sq)(1+s/q)}&
\frac{\theta(1-s^2)}{(1+sq)(1+s/q)}&0&0&-\frac{1}2\\
\frac{2\theta(1-s^2)}{(1+sq)(1+s/q)}&
\frac{\theta (1+s)^2(1+q)}{(1+sq)(1+s/q)(1-q)}&1&0&0\\
0&-2t_1t_2&0&0&0\\
0&0&0&0&2t_1t_2\\
-4t_1t_2&0&0&-1&\frac{\theta(1+q)}{1-q}
\end{matrix}
\right).
\end{gather*}

\subsection{Exponents and monodromy}

We first calculate the eigenvalues of the residues associated
to $\nabla$ around its singularities.  Notice in 
particular that, for integer values of the level $t_1+t_2$, these eigenvalues are integral.
We do not discuss the residue along the
first group of the singularities (\ref{first sing}) because they
are studied in \cite{okpanhilb}:

\begin{prop} The residue of the operator $M_{(1,\omega_i)}$ along
the divisor
$$s^{\alpha}(-q)^k=1$$
is zero if $(\alpha,\omega_i)=0$.  In general, it is  diagonalizable and:

\begin{itemize}
\item There exists $N=N(\alpha,k)$ such that
 $$ Spec( Res_{s^\alpha(-q)^k=1}M_{(1,\omega_i)})
\subset\{(t_1+t_2)l(k+l-1)\}_{l=1,\dots,N}.$$
\item If $k\ne 0$ then $N(\alpha,k)<m/k$.
\end{itemize}
For $M_{D}$, we have

$$
Res_{s^\alpha(-q)^k=1}M_D=k Res_{s^\alpha(-q)^k=1}
M_{(1,\omega_{i})}
$$
for any $i$ such that $(\alpha,\omega_i)\ne 0$.
\end{prop}

\begin{proof}
If $(\alpha,\omega_i)=0$, the vanishing statement is easy.  Otherwise,
let us introduce an auxilary operator

$$ G_l=e_{-\alpha}(-k)^le_{\alpha}(k)^l,\quad
G_1=Res_{s^\alpha(-q)^k=1}M_{(1,\omega_i)}.$$

The space $H_T(Hilb_m(\An))$ is the weight space $V[\Lambda-m\delta]$ of
the highest weight representation. Hence there exists
$N=N(\alpha,k)$ such that
$$ e_{\alpha}(k)^N V[\Lambda-m\delta]=0.$$
By the weight consideration, in the case $k\ne 0$ the number $N$
satisfies the proposed inequality. The following consequence of the
relations in the Lie algebra completes the proof:
$$(G_1-l(k+l-1)))G_lv=G_{l+1}v,$$
for any $v\in V$.
The claim on $M_{D}$ is immediate from our formulas.
\end{proof}

We sketch the construction of a level-raising operator
identifying the monodromies of $\nabla_{\An}(t_1,t_2)$ and $\nabla_{\An}(t_1-a,t_2-b)$
for $a,b \in \ZZ$.  Let $\mathcal{X}(a,b)$ denote the family of $\An$ surfaces over $\mathbf{P}^1$
associated to $\mathcal{O}(a)\oplus\mathcal{O}(b)$ by fiberwise quotient and resolution.  We set
$$\mathsf{S}(a,b)$$
to be the operator encoding its Gromov-Witten theory relative to the fibers over $0$ and $\infty$.
Under the assumption of the generation conjecture, $\mathsf{S}(a,b)$ can be computed using the techniques of \cite{gwan} and is given by rational functions in $q= e^{-iu},s_{1},\dots,s_{n}$.
A localization argument (and the GW/Hilb comparison) shows that
$$\nabla_{\An}(t_1,t_2)\mathsf{S}(a,b) = \mathsf{S}(a,b)\nabla_{\An}(t_1-a,t_2-b)$$
so that $\mathsf{S}(a,b)$ defines an intertwiner operator, identifying the monodromy provided
$t_1,t_2$ avoid a finite set of rational numbers which give poles for $\mathsf{S}$.

This generalizes the same construction for $\nabla_{\mathbb{C}^{2}}$ given in \cite{okpanhilb}.  
One can again show that at integer level $t_1+t_2 \in \ZZ$, the monodromy
is abelian and is semisimple provided $t_1$ and $t_2$ avoid the singularities
of $\mathsf{S}$.  At these values, the fundamental solution to the QDE 
is basically given by rational functions of $q,s_{1},\dots,s_{n}$. In more details,
the fundamental solution to the QDE is of the form:
$$ B(q,s_1,\dots,s_n)\cdot q^{M_D^{cl}}\prod_i s_i^{M^{cl}_{(1,\omega_i)}},$$
where $B(q,s_1,\dots,s_n)$ is a rational function of $q$ and $s_1,\dots,s_n$.

\subsection{Generalization to $D,E$ resolutions}

Finally, we briefly explain how the calculations of this paper can be extended to resolutions
$S_{\Gamma}$ 
of rational surface singularities associated to root lattices $\Gamma$ of type $D$ and $E$.  
The argument here 
was suggested to us by Jim Bryan and is based on an argument first used in \cite{bryan-katz-leung}.  As these singularites are no longer toric, there is only a $\mathbb{C}^*$-action on $S_{\Gamma}$.

Dimension considerations reduce the two-point calculation to studying the nonequivariant reduced virtual class.
Let $X_{0} \rightarrow \Delta$ be a smooth family of surfaces over the disk $\Delta$, obtained from a map
from $\Delta$ to the versal deformation space of $S_{\Gamma}$.
Its fiber over the origin is the resolved surface $S_{\Gamma}$ while all other fibers are given by affine surfaces; in particular, all compact curves on $X$ lie over the origin.  This family admits a deformation
$X_{z} \rightarrow \Delta$ so that for $z \ne 0$, there are a finite number of non-affine fibers
each isomorphic to $\mathcal{A}_{1}$.  These non-affine fibers are in bijection with positive roots $\alpha$ of $\Gamma$, and the smooth rational curve lies in the corresponding curve class $\alpha$.

For both $X_{0}$ and $X_{z}$, we can take associated family of Hilbert schemes
$$\Hb(X_{0}/\Delta), \Hb(X_{z}/\Delta)\rightarrow \Delta$$
which are again deformation equivalent.   An effective curve on 
$\Hb(X_0/\Delta)$ must be contained in  
$\Hb(S_{\Gamma})\subset \Hb(X_0/\Delta)$; similarly, an effective curve
on $\Hb(X_z/\Delta)$ must be contained in one of the copies of $\Hb(\mathcal{A}_1) \subset
\Hb(X_z/\Delta)$.  The key observation is that, for nonpunctual curve
classes $\beta$, we can identify the reduced virtual class on $\Hb(S_{\Gamma})$
with the relative virtual class (defined in the usual sense) of the family 
$\Hb(X_0/\Delta)$ over $\Delta$:
$$[\Mbar_{0,2}(\Hb(S_{\Gamma}),\beta)]^{\mathrm{red}} = 
[\Mbar_{0,2}(\Hb(X_0,\Delta),\beta)]^{\mathrm{vir}}.$$
A detailed proof can be found, for instance, in \cite{MP}.The analogous statement holds for $X_z$.
Along with deformation invariance between $X_{0}$ and $X_{z}$, this immediately
gives that only root curve classes contribute to reduced invariants and, in that case, the calculation is given by the case of  $\mathcal{A}_{1}$.  The result is that the two-point operator is again given by
the same expression in terms of the action of type $D,E$ affine algebras $\widehat{\mathfrak{g}}$
on Fock space:
$$\Theta_{+} =  \sum_{\alpha \in R^{+}}\sum_{k} :e_{\alpha}(k) e_{-\alpha}(-k): \log(1-s^{\alpha}q^{k}).$$

\vspace{+10 pt}
Department of Mathematics\\
Massachusetts Institute of Technology\\
Cambridge, MA 02139, USA\\
dmaulik@math.mit.edu

\vspace{+10 pt}
\noindent
Department of Mathematics\\
Princeton University\\
Princeton, NJ 08544, USA\\
oblomkov@math.princeton.edu

\end{document}